\documentclass[12pt,psfig]{article}
\usepackage{amsmath, amsthm, amssymb}
\usepackage{graphicx,epsfig}
\usepackage{amsfonts}
\usepackage{amscd}
\usepackage{mathrsfs}
\usepackage{enumerate}
\usepackage{latexsym}
\usepackage[all]{xy}
\newcommand{\B}{\mathscr{B}}
\newcommand{\F}{\mathscr{F}}
\newcommand{\LST}{\mathscr{L}}
\newcommand{\LS}{\mathscr{L}}
\newcommand{\PM}{\mathscr{P}^m}

\newcommand{\Pone}{\mathscr{P}^1}
\newcommand{\R}{\mathscr{R}}
\newcommand{\f}{\mathbb{F}}
\newcommand{\N}{\mathbb{N}}

\newcommand{\NC}{\mathscr{N}}
\newcommand{\ta}{\mathscr{T}}
\newcommand{\D}{\mathfrak{D}}
\newcommand{\NZ}{\mathbb{N}_0}

\newcommand{\Span}{{\rm span}_\f}
\newcommand{\ad}{{\rm ad}}
\newcommand{\coker}{{\rm coker\,}}
\newcommand{\rank}{{\rm rank}_\f}
\newcommand{\0}{\mathbf{0}}
\newcommand{\mr}{\mathbf{m}}
\newcommand{\m}{\mathbf{n}}
\newcommand{\n}{\mathbf{n}}

\newcommand{\ie}{{\em i.e.,} }
\newcommand{\eg}{{\em e.g.,} }
\newtheorem{thm}{Theorem}[section]
\newtheorem{cor}[thm]{Corollary}
\newtheorem{lem}[thm]{Lemma}
\newtheorem{prop}[thm]{Proposition}
\theoremstyle{definition}

\newtheorem{rem}[thm]{Remark}

\numberwithin{equation}{section}

\def\be {\begin{equation}}
\def\ee {\end{equation}}
\def\ba {\begin{eqnarray}}
\def\ea {\end{eqnarray}}

\newcommand{\dss}{\displaystyle}

\begin{document}
\baselineskip=18pt
\renewcommand {\thefootnote}{\dag}
\renewcommand {\thefootnote}{\ddag}
\renewcommand {\thefootnote}{ }

\pagestyle{empty}

\begin{center}
                \leftline{}
                \vspace{-0.00 in}
{\Large \bf The spectral sequences and parametric normal forms
} \\ [0.4in]

{\large Majid Gazor$^{*}$}
\footnote{$^*\,$Corresponding author. Phone: (98-311) 3913625; Fax: (98-311) 3912602;
Email: mgazor@cc.iut.ac.ir}

\vspace{0.15in}
{\small {\em Department of Mathematical Sciences,
Isfahan University of Technology
\\[-0.5ex]
Isfahan, 84156-83111, Iran
}}

and
\vspace{0.1in}

{\large Pei Yu}

\vspace{0.15in}
{\small {\em Department of Applied Mathematics,
The University of Western Ontario \\[-0.5ex]
London, Ontario, Canada \ N6A 5B7
}}

\vspace{0.2in}


\end{center}

\baselineskip=20pt
\vspace{0.4in}

\noindent
\rule{6.5in}{0.012in}

\vspace{0.1in}
\noindent
We generalize recent developments on normal forms and the spectral sequences method to make a foundation for parametric normal forms. We further introduce a new style and costyle to obtain unique parametric normal forms. The results are applied to systems of generalized Hopf singularity with multiple parameters. A different (new) version of this paper has been submitted for a possible publication in a refereed journal. 

\vspace{0.27in}
\noindent
{\it Keywords}: \ Spectral sequence; Parametric normal form;
Formal basis style and costyle.

\vspace{0.10in}
\noindent {\it 2000 Mathematics Subject Classification}:\,
34C20, 34A34, 16W5, 68U99.

\noindent
\rule{6.5in}{0.012in}

\vspace{0.2in}

\section{Introduction}\label{sec10}

This paper extends the recent developments of normal form theory (without parameters) via the spectral sequences method to parametric normal forms of differential equations. Spectral sequences method is one of the most elegant and powerful methods of computations. It has been applied in different branches of mathematics and has helped in solving many difficult computational problems. Arnold \cite{Arnold75,Arnold76} was the first to apply this method on singularity and normal form theory but it were not later used by others. Recently Sanders \cite{Sanders03,Sanders05} further elaborated the method and evidently encouraged a few researchers in using this method on normal form theory \cite{benderchur,Mord04}. Sanders \cite{Sanders03,Sanders05} and Murdock \cite{Mord04} mainly focus on normal forms of non-parametric vector fields while Benderesky and Churchill \cite{benderchur} applied the method on matrix normal forms. The later made their results based on an innovative general setting; by having a group structure acting on a vector space. We use their results to establish a foundation for parametric normal form of vector fields. The existing results, \ie \cite{benderchur,Mord04,Sanders03,Sanders05}, use neither time rescaling nor reparametrization while the use of these is the main new feature of this paper. Obviously, our method is also finely established for obtaining the spectral sequences of orbital equivalence of nonparametric vector fields.

We apply the method to obtain hypernormal forms for generalized Hopf singularities with multiple parameters; this has been done for such systems without parameters, see \eg \cite{AlgabaSur,baider,baiderchurch,belitskii,Gaeta1999,pwang,Sanders03,yy2002}. Considerable work has been done on singular Hopf bifurcations, that is, degenerate forms of the Hopf singularity that do not satisfy the conditions for a standard Hopf bifurcation but instead produce more complicated dynamics. We give a method for finding specialized hypernormal forms for Hopf singularities with specific degeneracies. The goals that follow are illustrated with this specific type of system, but are applicable more generally to any singularity.

A new significant feature of parametric normal forms is to enable us gaining the transformations between the original parametric system and the parametric normal forms; this is impossible via non-parametric normal forms \cite{PYuChen,yl}. Real life problems modeled by scientists and engineers usually involve parameters and obtaining these transformations is of fundamental importance in applications. The notion of normal form theory is to simplify nonlinear differential equations via a change of state variables such that the topological behavior of the system in the vicinity of a singular point remains unchanged, see \cite{ChowLiWang,Kuz,LiaYu,MurdBook}. Efficient usage of time rescaling with change of state variable simplify the systems further to their (simplest) orbital equivalence, see \eg \cite{AlgabaSur,Algaba,Zoladek02,Zoladek03}. Although using time rescaling efficiently is a key tool (a challenging task) for parametric normal forms, the obtained orbital equivalence (simplest) of parametric system may not be yet sufficiently simple for bifurcation and stability analysis. Thus, we also need to use reparametrization (change of parameters) to simplify the parametric systems beyond their orbital equivalence. Reparametrization requires also a new structure within the context of unique normal form theory. This is why we need to generalize the recent developments on the spectral sequences and their structures on normal form theory to accommodate time rescaling and reparametrization alongside with change of state variable.

\pagestyle{myheadings}
\markright{{\footnotesize {\it M. Gazor and P. Yu
\hspace{1.9in} {\it Spectral sequences and parametric normal forms}}}}

We provide an example of detailed calculations of hypernormal forms using the notations of the spectral sequences in section \ref{sec3}. Spectral sequences have been introduced for normal form calculations in \cite{benderchur,Mord04,Sanders03,Sanders05} but even in these papers, the spectral sequence method is usually kept in the background as a guide to the calculations, while the calculations themselves are done in the ordinary way. Here we show what it would look like to employ the notations of the spectral sequences in full. This has the advantage that the exact range of nonuniqueness is displayed at all times with every step. The associated disadvantage is the complexity of the notation. It seems worth while to have at least one example of this type of calculation in the literature, even if most users may choose to use a simpler notation and do some extra work on the side to keep track of the range of nonuniqueness when that is desired.

The rest of this paper is organized as follows. Section \ref{cohomol} introduces a new way of defining styles and costyles for normal forms (in addition to standard styles such as the inner product, sl(2), and simplified styles), \ie {\em formal basis style and costyle}.
Also the cohomology spectral sequences are briefly presented in section \ref{cohomol}. Section \ref{sec2} describes parametric state space, parameter space and parametric time space, and further some technical results are presented. Then, in section \ref{sec3}, the general theory and methodology are applied to obtain one of the two parametric normal forms (via formal basis style) presented in this paper for systems with multiple parameters and generalized Hopf singularity. Based on notions of invariant degenerate spaces, formal basis styles and costyles, in section \ref{alter}, the method of spectral sequences is distorted to formulate an alternative approach to obtain a more suitable parametric normal form for bifurcation and stability analysis. Section \ref{compare} provides some alternative normal forms which can be obtained via different approaches. These include different simplest parametric (and nonparametric) normal forms obtained by using and not using either or both of time rescaling and reparametrization. Finally conclusions are drawn in section \ref{sec6}.

\section{Formal basis style and the spectral sequences} \label{cohomol}

We first revisit filtration topology for presenting the notion of formal basis used here for determining a new {\em style} and {\it costyle} for unique parametric normal forms, see also \cite{Gazor,GazorYu}, and then discuss the notion of the cohomology spectral sequences.

Let $V= \prod^\infty_{i=1} V_i,$ where $V_i$ are finite dimensional vector spaces over the field $\f$ of characteristic zero. We call $V$ a graded vector space and each $V_i$ a homogenous space of grade $i.$ In order to make the paper more readable, we follow Murdock and Sanders \cite{MurdBook,MurdSandersbox} and denote $V= {\oplus}^\infty_{i=1}V_i$ in which its elements are represented by sums of a countable number of nonzero terms (recall the concept of formal power series). This notation should not be confused with the common direct sum of vector spaces whose elements can only be represented as a sum of finite many nonzero terms. One could use $\hat{\oplus}$ notation, whenever infinitely many terms are involved, to avoid this confusion, see \cite{baiderchurch,Gazor}. A countable ordered set (sequence) $\B=\{e_j|j\in \N\}\subseteq V$ is called a formal basis for $V$ if any element $v\in V$ can be uniquely represented by $v=\sum^\infty_{j=1}a_j e_j$ for $a_j\in \f.$

With every grading goes a filtration. Let $\F^kV=\{\sum^\infty_{i=k}v_i|v_i\in V_i\}$ and call $\F=\{\F^kV\}^\infty_{k=1}$ a filtration associated with the graded vector space $V= \oplus^\infty_{i=1} V_i.$ The filtration $\F$ induces a topology on $V$ by considering $\{v+\F^kV\}$ as an open local base for a vector $v\in V.$ The induced topology $\tau_\F$ from $\F$ is called filtration topology. Since the filtration topology is a first countable topology, its topology is completely understood by knowing its convergent sequences. In other words, with regards to the filtration topology we only need to know: a sequence $\{v_n\}\subseteq V$ converges to $v\in V$ if and only if for any $N\in \N$ there exists a natural number $k$ such that for any $n\geq k$ we have $v_n\in v+\F^NV.$ Thus, any sequence of the forms $\{\sum^n_{j=1}a_j e_j\}$ and $\{\sum^n_{i=1}v_i|v_i\in V_i\}$ converges respectively to $\{\sum^\infty_{j=1}a_j e_j\}$ and $\{\sum^\infty_{i=1}v_i|v_i\in V_i\}$ in the filtration topology; more precisely the filtration topology is the {\em finest topology} on a graded vector space in which all formal series are convergent. Note that the filtration $\F$ is Hausdorff, \ie $\bigcap_p \F^p V = \{0\},$ and exhaustive, \ie $\F^0 V = V$. This is why the spectral sequences induced by $\F$ is convergent in the sense of Cartan-Eilenberg, see the last paragraph on page (7). We refer the reader to \cite{Gazor,GazorYu} for more details on the filtration topology and formal basis (also formal decompositions). Note that the order of basis (and formal basis) is important for our formal basis style (costyle) and all graded structures defined in this paper are graded vector spaces over $\f.$ Also, note that the style of parametric normal forms used in this paper is different from other common styles such as inner product style, $sl(2)$ style or simplified style. Recall that a normal form {\em style} is a {\em rule} stating how to choose the complement spaces, see [20--22]. 
When the rule of how uniquely choosing complement spaces is applied to the transformation space rather than the space of vector fields, it is called costyle \cite{Mord08}. We call our style {\em formal basis style} and describe it in the following.

In order to express the formal basis style or in other words how we would choose a complement space $\NC$ for a vector subspace $W$ from $V,$ we assume $V$ has a formal basis $\B=\{e_n\}$ (or a finite ordered basis $\B=\{e_n\}^{\rank V}_{n=1}$). Then, we construct the complement space $\NC$ by inductively choosing the least natural number $n_k$ in which $e_{n_k}$ is not an element of $W\oplus\Span \{e_{n_i} |i < k\}.$ We continue this to either the process be terminated for a finite number $k$ ({\em i.e.,} $\NC= \Span \{e_{n_i}|i\leq k\}$) or obtain an infinite sequence $\{e_{n_i}\}$ ({\em i.e.,} $\NC= \Span \{e_{n_i}\}$). From now on, we refer to $\NC$ obtained in this manner by the unique complement space for $W.$ Succinctly stated, by setting an order on the formal basis we automatically determine which terms to be eliminated in priority; when we have some alternative terms to eliminate from the system in our normal form computation, those terms will be automatically eliminated in our style which are laid in lower orders in our formal basis. This is, indeed, one of the main purposes of using formal basis in this paper. Some common styles used in the literature are $sl(2)$ style, inner product style or to determine the complement spaces individually for each grade in all steps, see \cite{Mord08} for the original definition and discussion on style and costyle. Despite the other common styles, we believe it is an advantage for our approach to decide on priority of eliminating {\em certain terms} of the system well in advance of calculations by setting a fixed order on formal basis. For example, in this paper we practically give priority to the amplitude terms rather than phase terms of the same grade. To observe this, one should transform the normal form of the system in polar coordinates and see the relation $X$ and $Y$ with amplitude and phase terms, and then compare them with the rules (1-3) above in Remark \ref{order}. This way we do not need to decide which complement spaces to choose but we only calculate the complement spaces based on the order defined on the formal basis. It is imperative to distinguish (and not confuse) the usage of formal basis to set up our style (or costyle), in order to give priority to some terms over the others for their elimination (or for using them in the transformation maps), from the method of formal decompositions (described in \cite{GazorYuGen,GazorYu}) which employs a {\it chess-like} computation.

Throughout this paper, formal basis style is also applied for our notation $\pi_W$ to represent it as a unique projection on $W$. To formulate this, let $V$ be a vector space with a formal (an ordered) basis $\B=\{e_n\}.$ Then, there exists a unique complement space $\NC$ for $W$ ({\em i.e.,} $\NC\oplus W=V$) such that $\B\cap \NC=\{e_{n_i}\}$ is a formal (an ordered) basis for $\NC$ and for any $e_m\in\B,$ there exist a unique vector $w\in W$ and unique scalars $a_{n_{1}}, a_{n_{2}}, \cdots, a_{n_{N}} (n_N\leq m)$ satisfying $e_m= w+ \sum^N_{k=1}a_{n_{k}}e_{n_{k}}.$ Thus, $V$ can be naturally equipped with a unique projection $\pi_{W}$ from $V$ onto $W,$ {\em i.e.,} $\pi_{W}\circ \pi_{W}(V)= \pi_{W}(V)= W$ and for any $v=\sum_{i}a_{n_i}v_{n_i}+w\in V$ (where $\sum_{i}a_{n_i}e_{n_i}\in \NC$ and $w\in W$) we have $\pi_{W}(v)= w.$ Giving a simple example to illustrate this, consider $\pi_W(e_1+2e_2)$ for which $V=\f^2$ with standard ordered basis $\B=\{e_1, e_2\}$ and $W= \Span\{e_1+e_2\}.$ $\NC=\Span\{e_1\}$ uniquely satisfies our conditions (while $\NC=\Span\{e_2\}$ does not; $e_1$ can not be expressed by $e_1= w+ \sum_{n_k\leq 1}a_{n_{k}}e_{n_{k}}$ for a $w\in W,$ where we only have $n_1=2$) and thus, $\pi_W(e_1+2e_2)= 2e_1+2e_2.$

The style of parametric normal forms in the context of the spectral sequences comes from a unique choice to represent quotient spaces instead of the common complement spaces; {\em i.e,} for a vector subspace $W$ from $V,$ we need to explain how we choose a subspace $\NC$ from $V$ such that $\frac{V}{W}= \{v+W|v\in \NC\}.$ Indeed, any style (rule for choosing a complement space for $W$ in $V)$ can alternatively be applied for a unique choice of representing this, see also \cite{Sanders03,Sanders05}. Since in this paper we follow formal basis style, we are interested in defining a formal basis  compatible with $\B$ for quotient spaces of the given vector space $V.$ To do this, consider the basis $\{e_{n_i}\}$ for the unique complement space $\NC= \Span\{e_{n_i}\}$ obtained via formal basis style (described above) for $W.$ Then, $\{e_{n_i}\}$ with its inherited order from $\B$ builds up our formal basis for the quotient space $\frac{V}{W},$ {\em i.e,} $\B_{\frac{V}{W}}=\{e_{n_i}+ W\}.$ Since $\NC$ plays the same role as the complement spaces, for our convenience we call $\NC\subseteq V$ the unique complement space for the quotient space $\frac{V}{W}.$ We further borrow the notations and denote $\frac{V}{W}$ by $\NC,$ {\em i.e.,} $\frac{V}{W}= \NC.$ Thus, the unique complement space for a quotient space of $V$ is a vector subspace of $V$ with its own formal basis.

\begin{lem}\label{decom} (See also \cite[Proposition 1.3]{GazorYu}). Let $V$ and $W$ be vector spaces over $\mathbb{F},$ $V$ have a formal basis $\B=\{e_n\}^\infty_{n=1}$ or a finite ordered basis $\B=\{e_n\}^{\dim_\f V}_{n=1},$ and $\ta_0$ be a subspace of $V,$ where $d: W\rightarrow V$ is a linear map and $\pi_{\ta_0}d(W)= \0.$ Then, $d$ naturally induces a linear map $d^*,$ given in the short complex $0\rightarrow W \xrightarrow{d^*} \frac{V}{\ta_0}\rightarrow 0$ ($\frac{V}{\ta_0}$ denotes the quotient space of $V$ over $\ta_0$).
Then, there exist unique vector subspaces $\NC_1$ and $\ta_1\subseteq V$ such that
\begin{enumerate}
\item $\coker d^*\cong \frac{V}{\ta_1}= \frac{\NC_1+ \ta_1}{\ta_1}= \NC_1, \ta_0\subseteq \ta_1$ and $\NC_1\oplus \ta_1= V.$

\item \label{cond2} $\B\cap \NC_1= \{e_{n_k}\}$ is either an ordered basis or a formal basis
for $\NC_1.$

\item \label{cond3}
For any $e_m\in\B$ there exist a unique vector $\hat{w}\in V$ {\rm(}where
$\hat{w}+ \ta_0 = d^*(w)$ for some $w\in W$ and $\pi_{\ta_0}\hat{w}=\0${\rm)} and unique
scalars $a_{n_{1}}, a_{n_{2}}, \cdots, a_{n_{N}} (n_N\leq m)$
satisfying $e_m- (\hat{w}+\sum^N_{k=1}a_{n_{k}}e_{n_{k}})\in \ta_0.$
\item For any $v\in V$ in which $\pi_{\ta_0}(v)=0,$ there exists a vector $w\in W$ such that $v-d(w)\in \NC_1$ (This property
implies that $\NC_1$ fulfils the role of complement spaces).
\end{enumerate}
\end{lem}

The notation $\pi_W$ should be distinguished from $\pi_i(\mathbf{v});$ $\pi_i(\mathbf{v})= \mathbf{v}_i$ where $\mathbf{v}_i$ is the $i$th component of the vector $\mathbf{v}\in\f^n$ for $1\leq i\leq n.$ Now for our convenience, we present a quick review on the spectral sequences, see \cite{McCleary,Weibel} and for a more detailed discussion of this topic on normal forms see [8, 27, 28, 33--35].

Spectral sequences (cohomological type) in general are considered to be a page-sequence of left $R$-modules (where $R$ is a commutative ring with identity) which comes with a page-sequence of differentials and each page is the cohomology of the previous page. Consider a differential graded $R$-module $(A^*, d^*)$ where $d^*$ has a degree $+1,$ {\em i.e.,} $A^*= \oplus_{n} A^n,$ $$ d^*: A^*\rightarrow A^* \hbox{ with } d^{n}= d^*|_{A^{n}}: A^{n} \rightarrow A^{n+1}$$ and ${d^*}\circ d^*=0.$ 
Denote $H^n(A^*, d)$ for the cohomology of $d^*$ at grade $n,$ {\em i.e.,} $H^n(A^*, d^*)= \ker d^n/ {\rm Im} d^{n-1}.$ Now assume $A^*$ comes equipped with a filtration $\F$ compatible with the differential $d^*,$ that is, $$\cdots \subset \F^{p+1}A^*\subset \F^pA^*\subset \cdots \subset A^*$$ and $d: \F^p A^* \rightarrow \F^p A^*$ $\forall p\in \mathbb{Z}.$ Then, $(A^*, d, \F)$ is called a filtered differential graded $R$-module. For simplicity, assume that filtration comes from a grading structure, that is, $A^{*,*}$ is a bi-graded $R$-module $(A= A^{*,*}= {\oplus}_{q\geq0} A^{*,q}= \oplus_p A^{p,*}, A^{*,q}= \oplus_p A^{p,q}, \hbox{ and } A^{p, *}= {\oplus}_q A^{p,q})$ and $\F^qA^{*,*}= {\oplus}_{k\geq q} A^{*, k}.$ Thereby, the filtration is Hausdorff and exhaustive. 
We may express some isomorphisms with equality whenever confusion does not occur.

Spectral sequences are aimed at computing $H^{*}(A,d) = \oplus_q H^q(A, d),$ where $H^q$ denotes the cohomology of the differential $d$ at degree $q.$ Note that $H^{*}(A,d)$ can be interpreted in the context of normal form theory as the space of unique normal forms, when the graded differential $(A,d)$ is properly chosen. This is clearly demonstrated by Sanders \cite{Sanders03,Sanders05}. Although
\begin{equation}\label{HAd}
H^{*}(A, d)\cong {\oplus}_{p}\oplus_q \F^p H^q(A, d)/ \F^{p+1} H^q(A, d)\cong\oplus_q {\oplus}_{p} \F^p H^q(A, d)/ \F^{p+1} H^q(A, d),
\end{equation}
in principal it is not possible to compute $H^{*}(A,d)$ directly from this equation. This is why the spectral sequences are actually designed; to compute $E_\infty^{m,n}$ in a systematic approach. 
Obviously, it is important to properly define the differential, grading structure and the filtration in such a way that $E^{*,*}_r$ converges to $E_\infty^{*,*}\cong H^*(A, d).$ To describe the spectral sequences, we denote $$E_0^{m,n}= \F^{m} A^{m+n}/ \F^{m+1}A^{m+n}\cong A^{m+n, m}$$ and define the differential $d_0$ with degree $+1$ by $d_0^{m,*}=d|_{E_0^{m, *}}: E_0^{m, *}\rightarrow E_0^{m, *}.$ Since $E_{r}^{m,*}= \oplus_n E_{r}^{m,n}$ is isomorphic to a quotient space of $A^{*,m},$ we can inductively define the $r$th level differential of bi-degree $(r,1-r)$ by $$d_{r}^{m,*}= d_{|_{E_{r}^{m,*}}}, d_{r}^{m, n}:E_{r}^{m, n}\rightarrow E_{r}^{m+r, n+1-r}, \hbox{ and also }E_{r+1}^{m,n}= H^n (E_{r}^{m,*}, d_{r}^{m,*}).$$

The filtration $\F$ is called strongly convergent in the sense of Cartan-Eilenberg if the filtration is Hausdorff, exhaustive and we have
\begin{equation}\label{3Cartan}
H^*(A, d) \cong \displaystyle \lim_{\leftarrow p} H^*(A, d)/ \F^p H^*(A, d),
\end{equation} where $\dss \lim_{\leftarrow p}$ stands for the projective limit \cite[pp 69]{McCleary}. Since $$\displaystyle\lim_{\leftarrow p} H^*(A, d)/ \F^p H^*(A, d)\cong {\oplus}^\infty_{p=0} \F^pH^*(A, d)/ \F^{p+1} H^*(A, d),$$ our filtration automatically satisfies the condition $(\ref{3Cartan}).$ Therefore, the associated spectral sequence $E^{j,k}_r$ 
strongly converges to $H^*(A, d),$ that is, $$E_\infty^{j,k}\cong\F^j H^{j+k}(A, d)/ \F^{j+1} H^{j+k}(A, d),$$ according to \cite[Theorems 2.6, 3.2, 3.12]{McCleary}. 
Finally, a spectral sequence is said to collapse at $r,$ when $E^{*,*}_n= E^{*,*}_r$ for any $n\geq r.$ Thus, in order to obtain $E^{*,*}_\infty$ we just need to compute $E^{*,*}_r.$

\section{Parametric state space, parametric time space and parameter space}\label{sec2}

In this section, we present the algebraic structures of parametric state space, parametric time space and parameter space as well as their possible interactions. Note that this algebraic structure is designed for computation of parametric normal forms of systems associated with Hopf singularity. However, the methodology described here is general and can be applied to alternative algebraic structures suitable for other singularities.

The following presentation of our algebraic structures are recommended in part by J. Murdock, see also \cite{Gazor,GazorYu,MurdBook, pwang}. We begin with the most general $C^\infty$ system in two dimensions with vector parameter $\mu = (\mu_1, . . ., \mu_m)$ having a Hopf singularity at the
origin. When expanded in a formal power series such a system takes
the form (module flat functions)
\begin{equation}\label{or}
    \left(
      \begin{array}{c}
        \dot{x}\\
        \dot{y}\\
      \end{array}
    \right)= \left(
               \begin{array}{c}
                 y \\
                 -x \\
               \end{array}
             \right)+ \sum \left(
                        \begin{array}{c}
                          a_{jk\n} \\
                          b_{jk\n} \\
                        \end{array}
                      \right) x^jy^k\mu^\n
\end{equation}
where summation is taken over $j, k \in \NZ, \n= (n_1, \cdots, n_m)\in \NZ^m,$ $|\n| = |n_1| + \cdots  + |n_m|,$ $j+k+|\n|>1,$ and $j+k\geq 1.$ Introducing the complex variable $z = y + ix$
(to avoid later minus signs that arise if we use $x + iy),$ we find
\begin{equation*}
\dot{z}= iz+ \sum (b_{jk\n}+ ia_{jk\n}) \big(\frac{z-\overline{z}}{2i}\big)^j\big(\frac{z+i\overline{z}}{2}\big)^k\mu^\n
\end{equation*}
which can be expanded in the form
\begin{equation}\label{or2}
\dot{z}= iz+ \sum  A_{jk\n}z^j\overline{z}^k\mu^\n
\end{equation}
with $A_{jk\n} \in \mathbb{C}.$ We now consider the following (formal) system, defined on $\mathbb{C}^2,$ with variables $(z,w):$
\begin{equation}\label{or3}
\left(
  \begin{array}{c}
  \dot{z}\\
  \dot{w}\\
  \end{array}
\right) = i\left(
          \begin{array}{c}
          z   \\
          -w   \\
          \end{array}
          \right)+ \sum \left(
                \begin{array}{c}
                A_{jk\m}  \\
                B_{jk\m}  \\
                \end{array}
          \right)z^jw^k\mu^\m.
\end{equation}

With reality conditions $\overline{B_{jk\m}} = A_{jk\m},$ the equation (\ref{or3}) reproduces our original
system on the reality subspace of $\mathbb{C}^2$ defined by $w = \overline{z}$ (This is a
real vector space, that is, it is a subspace of $\mathbb{C}^2$ over $\mathbb{R}.$ For a complete
discussion of reality conditions in normal form theory, see \cite[pages 203-206]{MurdBook}). If we now set
\begin{equation*}
X_{jk} = \left(\begin{array}{c}z^jw^k\\w^jz^k\end{array}\right),Y_{jk} = \left(\begin{array}{c}iz^jw^k\\iw^jz^k\end{array}\right)
\end{equation*} and write $A_{jk\m} = \overline{B}_{jk\m}$ with $\alpha$ and $\beta$ real, our system on $\mathbb{C}^2$
takes the simple form
\begin{equation}\label{OR4}
\left(
\begin{array}{c}
 \dot{z} \\
 \dot{w} \\
\end{array}
\right)= Y_{10}+\sum \alpha_{jk\m} X_{jk}\mu^\m+\sum \beta_{jk\m} Y_{jk}\mu^\m
\end{equation}
where the reality conditions are now simply given by the $\alpha$ and $\beta$ coefficients lying in $\mathbb{R}.$ It is now easy to include in the discussion the case when the original system is complex (that is, $(x, y) \in \mathbb{C}^2).$ We simply take the field $\f$ to be either $\mathbb{R}$ or $\mathbb{C},$ and consider the right hand side of Equation (\ref{or2}) to be a vector field on $\mathbb{C}^2$ with coefficients in $\f$. The following analysis will apply to both the real and complex cases simultaneously.

From here on we take (\ref{OR4}) as the starting form for our analysis. This way we avoid the formulas for the coefficients of (\ref{OR4}) in terms
of those of (\ref{or}), which are of course rather complicated. Since the vector space span of vector fields of the form (\ref{OR4}) constitutes a Lie algebra $\LST$ under the Lie bracket $[u, v] = u' v - v' u\, (= {\rm Wronskian}(u, v)),$ that is, we set
\begin{equation*}
\LST=\{aY_{10}+\sum \alpha_{jk\m} X_{jk}\mu^\m+\sum \beta_{jk\m} Y_{jk}\mu^\m|a, \alpha_{jk\m}, \beta_{jk\m}\in \f, \m\in \NZ^m\}
\end{equation*} and call it {\it parametric state space}. We also use the subalgebras
\begin{equation*}
\LST_S=\{aY_{10}+\sum \alpha_{jk\0} X_{jk}+\sum \beta_{jk\0} Y_{jk}\}.
\end{equation*} which is called {\it state space without parameters} and \begin{equation*}
\LST_H=\{aY_{10}+\sum \alpha_{j+1, j\m} X_{j+1, j}\mu^\m+\sum \beta_{j+1, j\m} Y_{j+1, j}\mu^\m\},
\end{equation*} which is a Lie subalgebra of $\LST$ and is a result of normalization in the classical sense (without time rescaling and reparametrization and prior to hypernormalization). It is important to notice that these are not merely Lie algebras but
also graded Lie algebras. There are several ways to arrange the grading.
One is to treat these as multi-graded with one grading by $j + k -1$ (one less than the degree in $z$ and $w)$ and $p$ additional gradings
by $n_1,. . . ,n_m.$ The machinery for handling normalization with multiple
gradings is developed in \cite{Mord08} for the case of two gradings. Instead we combine these into a single grading $\delta$ defined
by $j + k +|\m|.$ Later, when considering specialized normal forms for degenerate Hopf singularities, we need a modified grading. The
general definition covering both cases is
\begin{equation}\label{grade}
\delta(X_{jk}\mu^\m)=\delta(Y_{jk}\mu^\m)= j+k-1+\alpha |\m|,
\end{equation} where $\alpha$ is a weight for the parameters. The weight $\alpha$ rearranges the terms in the graded Lie algebra and thus may result in different unique normal forms. This weight is taken to be one in \cite{GazorYu} for the codimension one Hopf singularity while in the next two sections, for the specialized normal forms of degenerate systems with {\it parametric dimension} $N_0,$ we choose $\alpha = 2N_0+1.$ In order to call $\LST$ a graded Lie algebra, it is necessary to check that $\delta([u, v]) = \delta(u) + \delta(v);$ this accounts for the $-1$ in the definition of $\delta.$ Thus, $(\LST, [\cdot, \cdot])$ is a $\mathbb{Z}$-graded locally finite parametric Lie algebra over $\mathbb{F}.$ For a simple notation to show the subspace of all homogenous terms of grade $k,$ we use a subindex $k,$ \eg $\LST_k, \LST_{S, k}, \LST_{H, k};$ so we do for later defined spaces, time and parameter spaces, \ie $\R_k, \PM_k.$

Let $\F$ be the filtration associated with the grading structure of $\LST.$ Elements of $\F^1\LST$ are called generators because they generate near-identity transformations used in normalization. Suppose that $u$ is a generator (so that $u\in \F^1\LST$ and therefore begins with quadratic terms) and $\phi$ is its time-one map. Writing $Z = (z,w),$ it follows that $\phi(Z) = Z +O(|Z|^2);$ such maps are called near-identity maps because they are close to the identity in a neighborhood of the origin. Let $v\in \LST$ be a vector field in $\LST,$ then $\phi_*(v),$ the push-forward of $v$ by $\phi,$ can be regarded as the same vector field $v$ expressed in modified coordinates. The easy way to compute this is by the formula
\begin{equation*}
    \phi_*(v)= {\rm exp\, ad}_u v,
\end{equation*} see e.g. \cite{baiderchurch,baidersanders,GazorYu,kw} for more details.

Once the parametric state space and its grading structure are settled down, it turns to define the formal basis. Let $\mathscr{B}= \{X_{ij}\mu^{\mathbf{n}}, Y_{ij}\mu^{\mathbf{n}}|\mathbf{n}\in \mathbb{N}^m_0, i,j\in \NZ, i+j>0\}$ be ordered in a sequence according to the following rules:
\begin{enumerate}
\item \label{c1}
The terms of lower grades are in lower orders, based on the grading function $\delta.$

\item \label{c2}
$Y_{ij}$ is before $X_{kl},$ when they have the same grade.

\item \label{c3}
Terms without parameter are before terms with parameter when they have the same grade.
\end{enumerate}

\begin{rem}\label{order}
Although the above rules are not sufficient to set up a unique order for $\B,$ they are sufficient that any fixed order satisfying
them will lead to a unique parametric normal form of a generalized
Hopf singularity; see sections \ref{sec3} and \ref{alter}.
\end{rem}

The reasons behind the conditions (\ref{c1}) and (\ref{c3}) are easy to observe; the condition (\ref{c1}) means the lower garde terms are in priority for elimination while the condition (\ref{c3}) is to omit the terms with parameters as much as possible even at the expense of some terms without parameters of the same grade. However, the condition (\ref{c2}) needs to be explored. Since $\mathscr{L}_{H}$ denotes the first level parametric normal form space when only the change of state variable is used (and also a component of the first level conormal form space) for generalized Hopf singularity, see Lemma \ref{firstlevel}, all the first level parametric normal forms belong to this space. This delivers a significant information about $X$ and $Y$ ($\in \mathscr{L}_{H}$) terms when they are depicted in polar coordinates, that is, any $X$ term in $\mathscr{L}_{H}$ is practically transformed into an amplitude term while $Y$ terms in $\mathscr{L}_{H}$ indirectly represent phase terms, see Corollary \ref{polarhopf}. This is why we put $Y$ terms before $X$ terms in the order of our formal basis (see the condition (\ref{c2}) above); in other words, the amplitude terms are in priority for elimination than phase terms of the same grade. Finally, it is important to mention that we choose an identical formal basis for parametric state space regardless of observing it as either the space of all vector fields (where normal forms live) or the space of all generators (where conormal forms lie). The later stems from a new notion, \ie costyle, in normal form theory.

The formal bases defined for parametric time space $\R$ and parameter space $\PM$ alongside with that of $\LS,$ are our rules in determining a unique complement space for any subspace within the transformation space in section (\ref{alter}). Murdock \cite{Mord08} is the first to raise the notion and call this kind of rule as {\em costyle}. Therefore, we follow him to call it by formal basis costyle. Indeed, any fixed costyle sets a rule to only have a unique choice (module unusable terms) for transformation solutions. Thus, it together with a fixed style make it possible to introduce unique (though far from being simplest) finite level normal forms, see \cite{Mord08} for further details on style and costyle. In principal costyle of normal forms is considered less important than style in the context of simplest normal form theory of systems without parameters. The importance of using costyle is indeed manifested in parametric normal forms, where obtaining the transformations are of the fundamental importance. In other words, costyle sets a rule for obtaining unique transformations (module unusable terms; belonging to the kernel of all maps at all steps) transforming a vector field to its simplest parametric normal form. Therefore, costyle can surprisingly play a direct role in obtaining {\it the simplest parametric normal forms} (not the transformations); this is demonstrated in section (\ref{alter}). To see this note that any costyle used in the section (\ref{sec3}) results in the same obtained parametric normal forms. In section (\ref{alter}), however, distorts the normal form computations by keeping some time terms for using them later; these terms do not belong to the kernel of the maps associated with step $N,$ but yet are used at steps of higher than $N.$ Therefore, the maps of step $N$ are restricted to some subspaces and their complement spaces (of vector space span of terms intended for use at step $N$) are left for higher level computations. In other words different costyles, in section (\ref{alter}), may lead to totally different normal forms; this signifies the role of costyle. One can foresee that implementation of the results, obtained in this way, for practical computations requires a chess like computation. Thus, it is essential to define a costyle coordinated with the style and our computations to gain our desired parametric normal form.

Similar to what we did with regards to the parametric state space, we now wish to do for time rescaling. Since time rescaling has to stand in the real numbers, we begin with the variable $z$ and parameters $\mu$ such that the time rescaling is given by
\begin{equation}\label{tim}
t= \tau (T_0+Y^T(z\overline{z}, \mu))= \tau T_0+ \tau\sum T_{k,\m} z^k\overline{z}^k\mu^\m,
\end{equation} where $T_0\neq 0$ and the sum is taken over $k\in \NZ$ and $\m\in \NZ^m,$ only if $k+|\m|\geq 1.$ Let $Z_k= z^k\overline{z}^k$ (in particular $Z_0=1$) and define the {\it parametric time space} by
\begin{equation}\label{tim}
\R=\f[[z\overline{z}, \mu]]= \{T_0+ \sum T_{k,\m} Z_k\mu^\m\}.
\end{equation} Therefore, the parametric time space is an integral domain and a vector space on $\f,$ where from now on $\f$ is the set of real numbers. Indeed, $\R$ is a locally finite graded vector space and also a graded ring; let $\B_{\R}= \{Z_i\mu^{\mathbf{n}}\}$ and define the grading function $\delta_\mathscr{R}:\B_\R\rightarrow\mathbb{Z}$ by
\begin{eqnarray}\nonumber
\delta_\mathscr{R}(Z_i\mu^{\mathbf{n}})=2i+ r\alpha,
&i,r\in \mathbb{N}_0.
\end{eqnarray} Note that the number $\alpha$ is the same for all three grading functions ({\em i.e.,} $\delta, \delta_\mathscr{R},$ $\delta_{\PM}$, for definition of $\delta_{\PM}$ see below) defined in this paper. The order of a formal basis for time rescaling plays a partial role in formal basis costyle. $\B_{\R}$ is our formal basis for $\R$ whose order in a sequence obeys the following rules (Remark \ref{order} is also true here):
\begin{enumerate}
\item The terms of lower grades are in lower orders.
\item The terms without parameter are before the terms with parameters, whenever they have the same grade.
\end{enumerate} Denoting $\R_k$ for the grade $k$ homogenous elements of $\R,$ we have $\R=\oplus^\infty_{k=0} \R_k$ and $\mathscr{R}_{k}\mathscr{R}_{l}\subseteq \mathscr{R}_{k+l}$ for any $k, l\in \NZ$ (\ie $\R$ is a graded ring). For our convenience we choose $T_0=1$ and thus, $Y^T\in \F^1\R$ is a generator for near identity time rescaling.

For any time rescaling generator $Y^T,$ there is a map $\phi^T_{Y^T}$ sending a vector field $v$ to $v+Y^Tv$ (indeed, the system $\frac{dz}{dt}=v$ is transformed to $\frac{dz}{d\tau}= v+Y^Tv);$ the multiplication ($Y^Tv$) follows the common multiplication of $Y^T$ and $v$ when both are presented as formal power series in terms of $z$ and $\overline{z}.$ The following two formulas are for our convenience:
\begin{eqnarray}\nonumber
Z_i\mu^{\mathbf{n}_{1}}X_{jk}\mu^{\mathbf{n}_{2}}&= & X_{(i+j)(i+k)}\mu^{\mathbf{n}_{1}+\mathbf{n}_{2}},\\\nonumber
Z_i\mu^{\mathbf{n}_{1}}Y_{jk}\mu^{\mathbf{n}_{2}}&= &Y_{(i+j)(i+k)}\mu^{\mathbf{n}_{1}+\mathbf{n}_{2}},
\end{eqnarray} for any $Z_i\mu^{\mathbf{n}_{1}}\in \B_\mathscr{R}$ and $X_{jk}\mu^{\mathbf{n}_{2}}, Y_{jk}\mu^{\mathbf{n}_{2}}\in \B.$ This product along with the grading structures on $\LST$ and $\R$ builds up $\LST$ as a graded module structure over the graded ring $\R,$ {\em i.e.,} $\mathscr{R}_{N_1}\mathscr{L}^2_{N_2}\subseteq \mathscr{L}^2_{N_1+N_2}.$

\begin{prop} $\F^1\mathscr{R}$ is a subgroup of $\mathscr{R}$-module filtration preserving automorphisms of $\LST,$ {\em i.e.,} $\phi^{T}_{\F^1\mathscr{R}*} \leq {\rm Aut}_\mathscr{R}(\LST)\leq {\rm Aut}_\f(\LST)$ and furthermore, $\phi^{T}_{{Y^{T}}*}(\F^n\LST)\subseteq \F^n\LST$ $(\forall n\in \NZ, {Y^{T}}\in \F^1\mathscr{R}).$
\end{prop}

Let $\f[[\mu]]$ denote for the integral domain of formal power series in terms of the parameters $\mu= (\mu_1, \mu_2, \cdots, \mu_m).$ Now the parameter space is defined by $$\PM= \mathbb{F}^m[[\mu]]= \{v| v= {(v_{1}, v_{2}, \cdots, v_{m})} \hbox{ where } v_{j}\in \Pone= \mathbb{F}[[\mu]]\,\,\, \forall j, 1\leq j\leq m\}.$$ For the given $\alpha\in \N$ and any $\mathbf{n}= (n_1, n_2, \cdots, n_m)\in \NZ^m$, we define a grading function on monomials $\mu^{\mathbf{n}}$ by $\delta_{\mathscr{P}^m}(\mu^{\mathbf{n}})= \sum^m_{i=1}n_i\alpha-\alpha.$ Denote $$\mathscr{P}^m_n= \Span\delta_{\mathscr{P}^m}^{-1}(n)\hbox{ for all }n\in \mathbb{N}_0\cup\{-\alpha\}.$$ Then, $\mathscr{P}^m_{-\alpha}=\mathbb{F}^m$ and
$\mathscr{P}^m={\oplus}^\infty_{n=0} \mathscr{P}^m_n$ is a locally finite graded vector space, where $\{\mu^{\mathbf{n}}\mathbf{e}_k|\mathbf{n}\in \mathbb{F}^m, k\leq m\}^\infty_{r=0}$ is a formal basis for $\mathscr{P}^m$ with lower grade terms (associated with $\delta_{\mathscr{P}^m})$ being ordered in a sequence before higher grade terms.

$\F^1 \mathscr{P}^m= \F^\alpha \mathscr{P}^m$ is called near identity parameter generators, since for any $Y^{P}(\nu)\in \F^\alpha \mathscr{P}^m$ we have a near identity reparametrization given by $\mu= \nu+ Y^{P}.$ Let $\phi^{P}_{{Y^{P}}*}(v(\mu))$ denote the vector field $v$ in terms of new coordinates $\nu.$ Then,
\begin{equation}\label{parametermap}
\phi^{P}_{{Y^{P}}*}(v(\mu))= \sum^\infty_{n=0}\frac{1}{n!}D^{n}_{\mu}(v, Y^{P}),
\end{equation} \noindent where $D^{n}_{\mu}(v,Y^{P})$ denotes the $n$th-order formal Frechet derivative of $v(\mu)$ with respect to $\mu$ but evaluated at $\nu$ and $(\overbrace{Y^{P}, Y^{P}, \cdots, Y^{P}}^{n\, {\rm times}}),$ see also \cite[Section 2]{GazorYu} and \cite{Kuz,LiaYu}. Obviously, the composition of any two near identity reparametrization is a near identity reparametrization and thus, $\F^\alpha \mathscr{P}^m$ forms a group acting on the parametric Lie algebra $\mathscr{L}^2.$
\begin{prop}
$\F^\alpha \mathscr{P}^m$ is a subgroup of $\f$-linear filtration preserving automorphisms of $\LST,$ {\em i.e.,} $\phi^P_{\F^\alpha\mathscr{P}^m*} \leq {\rm Aut}_\f \LST.$ Furthermore we have $\phi^{P}_{{Y^{P}}*}(\F^n\LST)\subseteq \F^n\LST$ $(\forall n\in \NZ$ and ${Y^{P}}\in \F^\alpha\mathscr{P}^m).$
\end{prop}

\begin{rem}\label{coex} For any $Y^S, v\in \LST,$ the condition ${\rm ad}^n_{Y^{S}} v=\mathbf{0}$ implies ${\rm ad}^k_{Y^{S}} v=\mathbf{0}$ $\forall k\geq n.$ Now let $m=1, Y^P\in \PM$ and $v_k(\mu)\in \LST_k.$ Then, the condition $D_{\mu}(v_k(\mu), Y^{P})= \mathbf{0}$ requires $D^n_{\mu}(v_k(\mu), Y^{P})= \mathbf{0}$ for any $n\geq 1.$ This is the main reason that in the computation of the simplest normal form for systems without parameter as well as the parametric normal form in which $m=1,$ the state and parameter operators are linear. On the contrary, consider
$m=2, \alpha=1,$
\begin{equation*}
v_2= X_{10}(\mu_1^2+\mu_2^2)\in \LST_2\hbox{ and }Y^{P}_2= \mu_2^2\partial_1- \mu_1\mu_2\partial_2\in \PM_2.
\end{equation*} Then,
\begin{equation*}
D_{\mu}(v_2(\mu),Y^{P}_2)= \mathbf{0} \hbox{ while } D^{2}_{\mu}(v_2(\mu), Y^{P}_2)= 2X_{10}\mu^2_2(\mu^2_1+\mu^2_2)\neq \mathbf{0}.
\end{equation*}
Furthermore, note that $D^{2}_{\mu}(v_2(\mu), Y^{P}_2)$ is not a linear operator in terms of $Y^{P}_2;$ see \eg $D^{2}_{\mu} (v_2(\mu), aY^{P}_2)= a^2D^{2}_{\mu}(v_2(\mu), Y^{P}_2).$ Therefore, parameter operators are not necessarily linear in general. This leads to an interesting structure which makes normal form computation complicated. Thus, it requires a new approach. However, in this paper we consider a condition on parameters ensuring the condition
\begin{equation}\label{linpar}
D_{\mu}(\sum^k_{i=0}v^{(i)}_{i}, Y^P)= \mathbf{0} \hbox{ implies } D^{n}_{\mu}(\sum^k_{i=0}v^{(i)}_{i}, Y^P)= \mathbf{0}\quad\forall n>1, k\in \mathbb{N}_0 (Y^P\in \F^\alpha\mathscr{P}^m).
\end{equation} Then, the parameter operators are linear. In other words, the linear part of parameter map (Equation (\ref{parametermap})) is injective and thus, there is no kernel term to be used for nonlinear parts of the parameter maps. This (in section (\ref{alter})) means that the parameters are in the right places for having a structurally stable amplitude equation (when we only think of topological behavior of the system, we may ignore the phase equation); this is observed in equation (\ref{stab}) where any of its small perturbation (upto $o(\rho^{2N_0}$)) does not change the system's topological behavior. Thus, we further assume a minimum number of parameters for such purpose, and also for our computation to simply get the simplest normal form, that is, $m=N_0$ in Lemma \ref{notmodifiedlem1} and Theorem \ref{specmain}. Therefore, there is neither unnecessary parameters in the system nor a need for extra unfolding parameters to get a structurally stable system, when the amplitude equation is only concerned. We refer to $m$ as the {\it parametric dimension} of the Hopf singularity system.
\end{rem}

For our convenience we say a {\it parametric} (generalized) Hopf singularity system has {\em parametric dimension} $m,$ when it has $m$ parameters (apart from the state variable and time) and $m$ is the least number of parameters necessarily present in the amplitude equation (of its parametric normal form is transformed in polar coordinates) to have the amplitude equation as a structurally stable equation. This means that there is no unnecessary parameters in the system, while the parameters already existed in the system are enough and are in the right places to make the amplitude equation (of its parametric normal form in polar coordinates) structurally stable. Therefore, many parametric systems may not have a parametric dimension; {\it either} for the system having more or less parameters than needed {\it or} for the parameters not being in the right places. We assume that all parametric systems considered in this paper have a parametric dimension. To treat systems without a parametric dimension, we can try to eliminate unnecessary parameters and unfold the system with unfolding parameters wherever they are needed. This is beyond the cope this paper. Although we do not try to address it, this has close links with the topological codimension.

The reader should note that a parametric system in real life problems may have more parameters than what we require here. Such extra parameters can be considered as control parameters and therefore, they are important in applications. This is the main reason for Yu and Leung \cite{yl} to consider and keep the extra parameters in their parametric normal forms. One may also try to reduce the number of parameters in their parametric normal form. 

The following two lemmas and Corollary \ref{minusone} play a key role in the method described in this paper.

\begin{lem}\label{3transf}
Let $\sigma$ be a permutation on the set $\{\phi^{P}_{\F^1 \mathscr{P}^m*},\phi^{T}_{\F^1\mathscr{R}*}, \phi^{S}_{\F^1\LST*}\},$ that is, $\sigma\in S_3,$ and ${\rm Aut}_\f(\LST)$ denote the filtration preserving $\f$-linear automorphism. Then, $\sigma(\phi^{P}_{\F^1 \mathscr{P}^m*})\times \sigma(\phi^{T}_{\F^1\mathscr{R}*})\cong \sigma(\phi^{P}_{\F^1 \mathscr{P}^m*}) \sigma(\phi^{T}_{\F^1\mathscr{R}*})= \sigma(\phi^{T}_{\F^1\mathscr{R}*})\sigma(\phi^{P}_{\F^1 \mathscr{P}^m*})\leq {\rm Aut}_\f(\LST).$ Furthermore,
$\phi^{P}_{\F^1 \mathscr{P}^m*} \times\phi^{T}_{\F^1\mathscr{R}*}\times \phi^{S}_{\F^1\LST*}\cong \sigma(\phi^{P}_{\F^1 \mathscr{P}^m*}) \sigma(\phi^{T}_{\F^1\mathscr{R}*})\sigma(\phi^{S}_{\F^1\LST*}) \leq {\rm Aut}_\f(\LST).$
\end{lem}

Note that the above lemma represents a semidirect product of subgroups, see our newer version of this paper for more details. This implies that none of the three transformations ({\em i.e., }state change of variable, reparametrization, and time rescaling) can be obtained from the other two. This is consistent with our claim that all three transformations are needed for parametric normal forms. It also explains why normal forms of non-parametric systems alongside with their orbital equivalence have been considered in the literature.

\begin{lem}\label{transgroup}
Let $n, p, q, l \in \N$ and $v\in \F^n\LST.$ Then, $$\pi_{\LST_s}\circ\sigma(\phi^P_{\F^l \mathscr{P}^m*}) \circ \sigma(\phi^T_{\F^p\mathscr{R}*})\circ \sigma(\phi^S_{\F^q\LST*})(v)= \{\pi_{\LST_s}(v)\}$$ $(\forall s, s< n+\min \{l,p,q\})$ and $\sigma\in S_{\{\phi^P_{\F^l \mathscr{P}^m*} ,\,\phi^T_{\F^p\mathscr{R}*},\,\phi^S_{\F^q\LST*} \}}.$ Furthermore, $\sigma(\phi^P_{\F^p \mathscr{P}^m*})\circ \sigma(\phi^P_{\F^q \mathscr{P}^m*})= \sigma(\phi^P_{\F^l \mathscr{P}^m*})$ where $l= \min(p,q)$ and $\sigma\in S_{\{\phi^P_{\F^k \mathscr{P}^m*} ,\, \phi^T_{\F^k\mathscr{R}*} ,\, \phi^S_{\F^k\LST*} \}}$ $(\forall k, k\in \NZ).$ In particular, for any $p\leq q,$ $\phi_{Y_1}\in \sigma(\phi^P_{\F^p \mathscr{P}^m*})$ and $\phi_{Y_2}\in \sigma(\phi^P_{\F^q \mathscr{P}^m*}),$ there exists $\phi_Y\in \sigma(\phi^P_{\F^p \mathscr{P}^m*})$ such that $\phi_{Y_2}\phi_{Y_1}= \phi_{Y}$ where $Y= Y_1$ ${\rm mode}$ $\F^q x$ $(x$ denotes $\mathscr{P}^m, \mathscr{R}$ or $\LST).$
\end{lem}
\begin{cor}\label{minusone} For any $q\in \NZ, p\in \N,$ and permutation $\sigma\in S_{\{\phi^P_{\F^p \mathscr{P}^m*} ,\,\phi^T_{\F^p\mathscr{R}*},\,\phi^S_{\F^p\LST*} \}}$ we have $$(\sigma\circ\phi^S_{\F^p\LST*}-1)(\F^q\LST)\subseteq \F^{p+q}\LST,$$ where $\sigma\circ\phi^S_{\F^p\LST*}-1$ is defined by  $(\sigma\circ\phi^S_{\F^p\LST*}-1)v= \sigma\circ\phi^S_{\F^p\LST*}v-v$ $(\forall v, v\in \LST).$
\end{cor}

\section{Parametric normal forms and the spectral sequences}\label{sec3}

In this section, we establish the setup for the method of spectral sequences for parametric normal forms of vector fields and apply it to a parametric generalized Hopf singularity with parametric dimension $N_0.$ Let $v\in \LST,$
$$A^{0,*}= \F^1 \mathscr{P}^m\times \F^1\mathscr{R}\times \F^1\LST \hbox{ and } \phi_{A^{0,*}} = \phi^P_{\F^1 \mathscr{P}^m*} \times \phi^T_{\F^1\mathscr{R}*}\times \phi^S_{\F^1\LST*} \leq {\rm Aut}_\f(\LST).$$ Then, define the map
\begin{eqnarray*}
\phi_{A^{0,*}}&\xrightarrow{\phi-1}& \LS \\
\phi_{(Y^P,Y^T,Y^S)}&\mapsto& \phi^P_{Y^P*}\circ\phi^T_{Y^T*}\circ\phi^S_{Y^S*}(v)-v,
\end{eqnarray*} where $\phi_{(Y^P,Y^T,Y^S)}=(\phi^P_{Y^P*}, \phi^T_{Y^T*},\phi^S_{Y^S*})$ and ${(Y^P,Y^T,Y^S)}\in A^{0,*}.$ Thus, our goal is to find $Y^{(\infty)}=(Y^P, Y^T, Y^{S})$ and $v^{(\infty)}= \phi_{Y^{(\infty)}}(v)$ such that $v^{(\infty)}$ represents the unique parametric normal form. This map is not linear, thus similar to Benderesky and Churchill \cite{benderchur}, we, instead, work with its initially $\f$-linear map, that is,
\begin{eqnarray*}
0\rightarrow&A^{0,*} &\xrightarrow{d}  \LS \rightarrow 0,
\end{eqnarray*} defined by
\begin{equation}\label{differential}
d(Y) = D_\mu (v)Y^P+ Y^Tv+ {\rm ad}_{Y^S}v, \hbox{ where } Y= (Y^P, Y^T, Y^S).
\end{equation}
Let $A^{1,*}= \LS,$ $A^{*,*}= A^{0,*}\oplus A^{1,*}$ and $A^{j,k}= \{\mathbf{0}\}$ $(\forall j\neq 0,1).$ Then, $(A^{*,*}, d, \F)$ is a locally finite graded filtered differential. Obviously, $d$ depends on $v$ and so does the spectral sequence $E^{j,k}_r.$ We sometimes denote the differential defined in equation (\ref{differential}) by $d(v, Y),$ {\em i.e.,} $d(Y)= d(v, Y).$
\begin{rem}
Sanders \cite{Sanders03} considered the spectral sequence $E^{j,k}_r$ by allowing $v^{(r)}$ to be updated during the process of normal form, {\em i.e.}, $E^{j,k}_r= E^{j,k}_r(v^{(r-1)}).$ This is one of the most common and convenient approach in normal form theory. Benderesky and Churchill's result \cite{benderchur} implies that updating $v^{(r)}$ does not change $E^{n,-n+1}_r$ (by proving that $E^{j,k}_r(v)$ is invariant under the group orbit of $v$). Therefore, both approaches are equivalent. One should note that \cite{benderchur} applied the method in the context of matrix normal form theory, that is why they chose not to update the differentials, indicating that ``the spectral sequences do not admit useful morphisms as one varies $v$'', {\em e.g.,} updating $v$ into $v^{(r)}.$ This is also true in a sense for normal form of vector fields, see Lemma \ref{risr} 
and \cite[Theorem 6.11]{benderchur}. It, however, is evident that updating the differential substantially reduces the complexity of computations, see \eg \cite[Example 2.4.4.]{Gazor}. 
This is the main reason which led Sanders to come up with his innovative idea. Therefore, we also follow his idea of using the converging differentials $d_r= d^{*,*}_r(v^{(r)})$ at each new level of the spectral sequence. Thus, the results~\cite[Theorem 2.6, 3.2, 3.12]{McCleary} are still valid here and so is the argument given in the last paragraph of section \ref{cohomol}.
\end{rem}

\begin{lem}\label{r1}
For any $v\in \LS,$ there exists an automorphism associated with $Y^{(1)} \in A^{0,*}$ such that it uniquely sends $v$ into $\widetilde{v}^{(1)}= \phi_{Y^{(1)}}(v)\in {\oplus}^\infty_{n=0} \NC_n^{(1)}\subset \LS,$ where $$E^{-n,n+1}_1= \frac{\LS_n}{\ta_n^{(1)}} = \frac{\NC_n^{(1)}+ \ta_n^{(1)}}{\ta_n^{(1)}}= \NC_n^{(1)}\quad(\forall n, n\in \NZ)$$ and the formal basis style is used for the complement spaces $\NC_n^{(1)}.$
\end{lem}
\begin{proof} 
Following Lemma \ref{decom}, there exists a unique vector space $\NC_n^{(1)}$ satisfying $$E^{n,-n+1}_1= \frac{\F^{n} \LS}{\F^n \LS\cap d(\F^{n}A^{0,*})+\F^{n+1}\LS} 
= \frac{\LS_n}{\pi_{\LS_n}\!\circ d A^{0,n}}= \frac{\NC_n^{(1)}+ \ta_n^{(1)}}{\ta_n^{(1)}}.$$ Thus, there exists $Y_n= (Y^P_n, Y^T_n, Y^S_n)\in A^{0,n}$ such that $v^{n-1}_n+ \ta_n^{(1)}= v^n_n+ \ta_n^{(1)},$ where $v^n_n\in \NC_n^{(1)}$ and $v^{n}_n= v^{n-1}_n+ \pi_{\LS_n}d Y_n.$ Therefore, $v^n= \Phi_n(v^{n-1})= \phi^P_{Y^P_n*}(\phi^T_{Y^T_n*}(\phi^S_{Y^S_n*}(v^{n-1})))=  v^0_0+ v^1_1+ \cdots + v^{n-1}_{n-1}+ (v^{n-1}_{n}+\pi_{\LS_{n}} d Y_n)+ \cdots,$ where $\pi_{\LS_{n}}v^n= v^n_n= v^{n-1}_{n}+ D_\mu v^{(0)}_0Y^P_{n}+Y^T_{n}v^{(0)}_0+ \ad_{Y^S_{n}}v^{(0)}_0.$ By Lemma \ref{3transf}, there exists $Y_{(n)}= (Y^P_{(n)}, Y^T_{(n)},Y^S_{(n)})\in A^{0,*}$ such that $v^n= \Phi_n\cdots \Phi_2\Phi_1(v^{0})=\phi^P_{Y^P_{(n)}}\circ\phi^T_{Y^T_{(n)}}\circ\phi^S_{Y^S_{(n)}}(v^{0}).$ Lemma \ref{transgroup} implies that $Y_{(n)}$ converges in filtration topology to an element $Y^{(1)}$ from $A^{0,*}$ and thus, $v^n$ is also convergent to $\widetilde{v}^{(1)},$ {\em i.e.,} $\widetilde{v}^{(1)}= \Phi_{Y^{(1)}}v^{(0)}.$

The proof is complete.
\end{proof}
We follow Murdock \cite[Section 5]{Mord04} to call $\widetilde{v}^{(1)}$ (and $\widetilde{v}^{(r)}$) the first (the $r$th) level extended partial parametric normal form.

\begin{lem} \label{risr} For any $v\in \LS,$ there exist parameter solution $Y^P\in \F^1\PM$, time solution $Y^T\in \F^1\R,$ and state solution $Y^{S}\in \F^1\LS$ such that their associated automorphisms transform $v$ into a unique vector field (the $r$th level extended partial parametric normal form) $$\widetilde{v}^{(r)}= \Phi_{(r)}(v) \in {\oplus}^\infty_{n=0} \NC_n^{(r)} \qquad (\Phi_{(r)}= \phi^P_{Y^P*}\!\circ\phi^T_{Y^T*}\!\circ\phi^S_{Y^S*}),$$ where $\NC_n^{(r)}$ follows Lemma \ref{decom}, {\em i.e.,} $E^{-n,n+1}_r= \LS_n/\ta_n^{(r)} = \frac{\NC_n^{(r)}+ \ta_n^{(r)}}{\ta_n^{(r)}}=\NC_n^{(r)}$ $\forall n\in \NZ.$
\end{lem}

\begin{proof}
By Lemma \ref{r1} there exists a unique (since $\widetilde{v}^{(1)}$ follows formal basis style) 
vector field $\widetilde{v}^{(1)}= \phi_{Y^{(1)}}(v)\in {\oplus}^\infty_{n=0} \NC_n^{(1)}= \NC^{(1)},$ where ${\rm Total}^{1} (E_1^{*,*})= \frac{\NC^{(1)}+ \ta^{(1)}}{\ta^{(1)}}= \NC^{(1)}.$ 
Now define a new differential $d_1= d_1^{n,-n}$ induced by $d(\widetilde{v}^{(1)}, Y_n)= D_\mu (\widetilde{v}^{(1)})Y^P_n+ Y^T_n\widetilde{v}^{(1)}+ {\rm ad}_{Y^S_n}\widetilde{v}^{(1)}$ (where $Y_n= (Y^P_n,Y^T_n,Y^S_n)),$ {\em i.e.,}
\begin{eqnarray}
\ker d_0^{n,-n}&\xrightarrow{d_1^{n,-n}}& \coker\, d_0^{n+1,-n-1}
\\\nonumber
Y_n+ Z^{n+1,-n-1}_0&\mapsto&d(\widetilde{v}^{(1)}, Y_n)\quad (\hbox{mode} \,\, \F^{n+1} \LS\cap d(\F^{{n+1}}A^{0,*})+\F^{n+2}\LS).
\end{eqnarray} Since $\widetilde{v}^{(1)}_0= v^{(0)}_0,$ the above map is well-defined. Thus,
$E^{n,-n}_2= Z^{n,-n}_2/Z^{n+1,-n-1}_{1}
,$ where $Z^{n,-n}_2= \F^{n} A^{0, *}\cap d^{-1}_1\F^{n+2}A^{1,*}$ and $Z^{n+1,-n-1}_{1}= \F^{n+1} A^{0, *}\cap d^{-1}_1\F^{n+2}A^{1,*}.$ Lemmas \ref{r1} and \ref{3transf} prove that there exists an automorphism $\Phi_{(2)}$ which sends $v$ into $\widetilde{v}^{(2)}= \Phi_{(2)}(v) \in {\oplus}^\infty_{n=0} \NC_n^{(r)},$ where
$\widetilde{v}^{(2)}$ is the second level extended partial parametric normal form of $v.$
This confirms our claim for $r=2,$ and therefore the proof is finished by mathematical induction.
\end{proof}

By the above Lemma there exist state solution $Y^{S, n},$ parameter solution $Y^{P, n}$ and time solution $Y^{T, n}$ such that
$$v^{(n)}= \Phi_n(v^{(n-1)})= \phi^T_{Y^{T, n}*}\circ\phi^P_{Y^{P, n}*}\circ\phi^S_{Y^{S, n}*}(v^{(n-1)})= \sum^\infty_{k=k_0} v^{(n)}_k, \,\,\, v^{(n)}_k\in \NC_k \quad \forall k\leq n.$$
$\{v^{(n)}\}^\infty_{n=0}\subset \LS$ is a convergent sequence to a vector field $v^{(\infty)}\in \LS$ with respect to filtration topology. We call $v^{(\infty)}$ an infinite level (order or unique) parametric normal form. The above argument leads to the following theorem.

\begin{thm}\label{generalthm} Let $v= v^{(0)}= \sum^\infty_{n=k_0}v^{(0)}_n\in \LS,$ where $v^{(0)}_n\in \LS_n.$ Then, there exist a sequence of near-identity maps $\{\Phi_n\}_{n}$ which transforms $v^{(n)}$ to $v^{(n+1)}= \Phi_{n+1}(v^{(n)})$ $\forall n\in \NZ$ in which $v^{(n)}$ converges {\rm(}with respect to filtration topology{\rm )} to an infinite level parametric normal form $v^{(\infty)}= \sum^\infty_{r=k_0}v^{(\infty)}_r,$ where $v^{(\infty)}_r\in \NC_r^{(r)},$ $E^{-r,r+1}_\infty= E^{-r,r+1}_r= \LS_r/\ta_r^{(r)} = \frac{\NC_r^{(r)}+ \ta_r^{(r)}}{\ta_r^{(r)}}= \NC_r^{(r)}$ $\forall r\in \NZ$ and $\NC_r^{(r)}$ follows Lemma \ref{decom}. Furthermore, there exist $Y_{(n)}$ and $Y^{(\infty)}= (Y^P, Y^T, Y^S)\in A^{0,*}$ such that $\Phi_{(n)}= \Phi_n\circ\cdots \Phi_2\circ \Phi_1$ is associated with $Y_{(n)},$ $Y_{(n)}$ converging to $Y^{(\infty)}$ with respect to filtration topology, and $v^{(\infty)}= \phi^T_{Y^{T}*}\circ\phi^P_{Y^{P}*}\circ\phi^S_{Y^{S}*}(v).$
\end{thm}

\begin{proof}
For any $p, r\in \NZ,$ $Z^{p,-p+1}_r= \F^{p}A^{1, *}= Z^{p,-p+1}_\infty,$ while assuming $p- r\leq 1,$ we have $B^{p,-p+1}_{r} 
= \F^p A^{1, *}\cap d(\F^{1}A^{0, *}) =B^{p, -p+1}_\infty.$ Then, $$E^{r,-r+1}_r= \frac{Z^{r,-r+1}_r}{Z^{r+1,-r}_{r-1}+ B^{r,-r+1}_{r-1}}= E^{r,-r+1}_\infty \hbox{ and } {\rm Total}^{1} (E_\infty^{*,*})= {\oplus}^\infty_{r=0}E^{r,-r+1}_r.$$ Since $A^{0, *}$ is locally finite, the rest of the proof follows Lemma \ref{risr} and $v^{(\infty)}_r= \widetilde{v}^{(r)}_r,$ where $\widetilde{v}^{(r)}= \sum^\infty_{n=0} \widetilde{v}^{(r)}_n$ denotes the $r$th level extended partial parametric normal form.
\end{proof}

In the rest of this section we apply the method of spectral sequence described above to obtain a parametric normal form for generalized Hopf singularity of parametric dimension $N_0$. The following lemma presents the first level parametric normal forms of a system with generalized Hopf singularity.

\begin{lem}\label{firstlevel}
There exists $({Y^T}, 0,{Y^S})\in {\oplus}^\infty_{n=1} A^{0,n}$ which uniquely transforms $v$, given by
\begin{eqnarray}\label{originalvect}
v= v^{(0)}&= &Y_{10}+\sum_{i+j+r=2,\, |\mathbf{n}|=r,\, i+j\geq 1}^\infty \!\!\! a^{(0)}_{ij\mathbf{n}}X_{ij}\mu^{\mathbf{n}}+
\sum_{i+j+r=2,\, |\mathbf{n}|=r,\, i+j\geq 1}^\infty \!\!\! b^{(0)}_{ij\mathbf{n}}Y_{ij}\mu^{\mathbf{n}},
\end{eqnarray}
into the first level extended partial parametric normal form
\begin{equation}\label{1extend}
\widetilde{v}^{(1)}= \phi^{P}_{Y^P}\circ\phi^{T}_{Y^T}\circ\phi^{S}_{Y^S}(v) \in Y_{10}+ {\oplus}^\infty_{n=1} \NC_{n}^{(1)},
\end{equation} where $\NC_{n}^{(1)}= \Span \{X_{(k+1)k}\mu^{\mathbf{n}}| n= 2k+r\alpha,\mathbf{n}\in \NZ^m, |\mathbf{n}|= r\}\subset \LST_{H, n}$ $(\forall n\in \N).$ Furthermore, the first level parametric conormal form space is ${\oplus}^\infty_{n= 1} \{0\}\times \PM_n\times \LST_{H, n}.$
\end{lem}
\begin{proof} Let $\LS= \Span \{Y_{10}\}\oplus^\infty_{k=1} \LST_k$ and $n\in \N.$ Then, $Z^{n,-n+1}_1= \F^{n}\LS,$ $Z^{n+1,-n}_{0}= \F^{n+1} \LS$ and $B^{n,-n+1}_{0}= \F^n \LS\cap d_0(\F^{n}A^{0,*}),$ where $d_0(v^{(0)}, (Y^T, Y^P, Y^S))= D_\mu v^{(0)} Y^P+ Y^T v^{(0)}+ \ad_{Y^S}v^{(0)}.$ Therefore,
\begin{eqnarray*}
B^{n,-n+1}_{0}+ Z^{n+1,-n}_{0}&=& D_\mu Y_{10} \mathscr{P}^m_n+ \mathscr{R}_n Y_{10}+ \ad_{Y_{10}}\LS_n+ \F^{n+1} \LS \\
&=& \mathscr{R}_n Y_{10}+ \ad_{Y_{10}}\LS_n+ \F^{n+1} \LS\\
&=& \Span \{Y_{(k+1)k}\mu^{\mathbf{n}}| n= 2k+r, k, r\in \NZ\} + \LS_{H^c, n} + \F^{n+1} \LS\\
&=& \ta_n.
\end{eqnarray*} Thus, by Lemma \ref{decom}, $E^{n,-n+1}_1 = \frac{\F^n\LS}{\ta_n^{(1)}}
= \frac{\NC_n^{(1)}+ \ta_n^{(1)}}{\ta_n^{(1)}},$ where $$\NC_n^{(1)}= \Span \{X_{(k+1)k}\mu^{\mathbf{n}}| n=2k+r, \alpha |\mathbf{n}|= r, k, r\in \NZ\}.$$
Since $Z^{n,-n}_1 
= \{0\}\times \PM_n\times \LS_{H, n}+\F^{n+1}A^{0,*},$  $Z^{n+1,-n-1}_{0} 
= \F^{n+1} A^{0, *}$ and $B^{n,-n}_{0}= \{\mathbf{0}\}$,
$E^{n,-n}_1= \{0\}\times \PM_n\times \LS_{H, n}+\F^{n+1}A^{0,*}/\F^{n+1}A^{0,*}.$ Then,
$$ {\rm Total}^{1} (E_1^{*,*})= {\oplus}^\infty_{r=0} E^{r,-r+1}_1\cong {\oplus}^\infty_{r=0} \NC_r^{(1)}\subset \LS.$$ Besides,
${\rm Total}^{0} (E_1^{*,*})= {\oplus}^\infty_{n=1}E^{n,-n}_1\cong {\oplus}^\infty_{n=1} \{0\}\times \PM_n\times \LS_{H, n}.$

The rest of the proof is straightforward by Lemma \ref{r1}.
\end{proof}

The parametric dimension of the vector field $v$ is $N_0$ if and only if for any natural number $r,$ $\pi_{\f X_{(i+1)i}}\tilde{v}^{(r)}=0$ $(\forall i<N_0)$ and $\pi_{\f X_{(N_0+1)N_0}}\tilde{v}^{(r)}\neq 0.$ In other word, the parametric dimension of the system $\tilde{v}^{(i)}$ ($i\neq 0$) is $N_0$ if and only if its representation in the polar coordinates has no amplitude terms of grade less than $2N_0+1$ while has a nonzero term of $\rho^{2N_0+1}.$ When $N_0>1,$ the system is associated with a generalized Hopf singularity. Let us denote
\begin{equation}\label{matA}
A^{(1)}= [D_\mu \widetilde{v}^{(1)}(\mu= \mathbf{0})]_{\{X_{(k+1)k}\}^{N_0-1}_{k=0}}
\end{equation} for the unique matrix representation of linear map $D_\mu \widetilde{v}^{(1)}$ at $\mu= \mathbf{0}$ on the finite dimensional vector space $W= \Span\{X_{(k+1)k}\}^{N_0-1}_{k=0}$ in terms of its ordered basis $\{X_{(k+1)k}\}^{N_0-1}_{k=0}.$
When ${\rm rank}_{\mathbb{F}}(A^{(1)})$ equals the parametric dimension ($N_0$) of the system, we say $v$ is generic with respect to parameter or parameter generic. Thus, parameter generic merely means that the parameters are in the right places for having a structurally stable amplitude equation. Note that these assumptions are essentially useful in evaluating the converging differentials $d_{r}$ and the individual level spaces of the  spectral sequences. We omit the proof of the following lemma and Theorem \ref{specmain} for brevity.
\begin{lem} \label{notmodifiedlem1}
Assume the hypothesis given in Lemma \ref{firstlevel} and that the vector field $v$ is parametric generic and of parametric dimension $N_0$, $m=N_0,$ and $\alpha=2N_0+1.$ Then,
\begin{equation}\label{totalN0}
{\rm Total}^{1} (E_{2N_0}^{*,*})= {\oplus}^\infty_{r=0} E^{r,-r+1}_{2N_0}\cong {\oplus}^\infty_{r=0} \NC_r^{(2N_0)}\subset \LS,
\end{equation} where $\NC_{n}^{(2N_0)} = \Span \{X_{(k+1)k}| n= 2k\}\subset \LST_{H, n}$ for any $n, 1\leq n\leq 2N_0,$ for some $\sigma\in S_{N_0},$
$$\NC_{n}^{(2N_0)} = \Span (\{X_{k(k-1)}\mu_{\sigma(k)}| n= 2k+2N_0-1\}\cup\{X_{(k+1)k}| n= 2k\})$$
$\forall n, 2N_0<n\leq 4N_0+1,$ and $\NC_{n}^{(2N_0)}= \Span \{X_{(k+1)k}\mu^{\mathbf{n}}| k\geq N_0, n= 2k+r\alpha,\mathbf{n}\in \NZ^{N_0}, r= |\mathbf{n}|\}$ $\forall n\geq 2N_0.$ Furthermore, the $2N_0$th level conormal form space is ${\rm Total}^{0} (E_{2N_0}^{*,*})= {\oplus}^\infty_{r=1}E^{n,-n}_{2N_0}\cong {\oplus}^\infty_{r=1} \{0\}\times \{\mathbf{0}\}\times \LS_{H, n}.$

$E_{r}^{*,*}$ is strongly convergent to $E_{\infty}^{*,*}$ (more precisely, the filtration $\F$ is strongly convergent in the sense of Cartan-Eilenberg) and $E^{*,*}_r$-terms collapses at $r= 2N_0+1,$ {\em i.e.,}
\begin{equation}\label{totalinfty}
{\rm Total}^{1} (E_{\infty}^{*,*})= {\rm Total}^{1} (E_{2N_0+1}^{*,*})= {\oplus}^\infty_{r=0} E^{r,-r+1}_{2N_0+1}\cong {\oplus}^\infty_{r=0} \NC_r^{(2N_0+1)}\subset \LS,
\end{equation} where
\begin{equation}\label{componenttotalinfty}
\NC_{n}^{(2N_0+1)} = \NC_{n}^{(2N_0)} \quad (1\leq n\leq 2N_0), \NC_{n}^{(2N_0+1)}= \{\mathbf{0}\}\quad (2N_0< n< 4N_0),
\end{equation} and
\begin{equation}\label{componenttotalinfty1}
\NC_{n}^{(2N_0+1)}= \Span \{X_{(2N_0+1)2N_0}\mu^{\mathbf{n}}| n= 4N_0+r, \mathbf{n}\in \NZ^{N_0}, |\mathbf{n}|= r\}\quad (\forall n, 4N_0\leq n).
\end{equation}
\end{lem}
\begin{cor}
Let the vector field $v$ given by equation (\ref{originalvect}) be associated with a parametric generic system of $N_0$-parametric dimension and $m=N_0.$ Then, there exists $Y= (Y^T, Y^p, Y^S)\in {\oplus}^\infty_{n=1} A^{0,n}$ such that $\phi_Y= \phi^{P}_{Y^P_0}\circ \phi^{P}_{Y^P}\circ\phi^{T}_{Y^T}\circ\phi^{S}_{Y^S}$ ($\phi^{P}_{Y^P_0}$ stands for an invertible linear reparametrization) uniquely transforms $v$ into the infinite-level parametric normal form
\begin{eqnarray}\label{inftyextend}
v^{(\infty)}&=& Y_{10}+ a^{(\infty)}_{(N_0+1)N_0\mathbf{0}}X_{(N_0+1)N_0}+ \sum^\infty_{|\m|=0, \mathbf{n}\in \N_0^{N_0}} a^{(\infty)}_{(2N_0+1)2N_0\mathbf{n}} X_{(2N_0+1)2N_0}\mu^{\mathbf{n}}\\\nonumber
&+& \sum^{N_0}_{k=1} X_{k(k-1)}\mu_{\sigma(k)} \qquad\qquad\qquad\qquad\qquad(\hbox{for some } \sigma\in S_{N_0}).
\end{eqnarray}
\end{cor}
\begin{proof} The proof is straightforward following Lemma \ref{notmodifiedlem1} and Theorem \ref{generalthm}.
\end{proof}

The key property of the spectral sequence in the context of normal form theory is that transformation spaces in each new level (cohomology of previous page) coincide with the kernel of the maps associated with the previous level. This is the main reason that eliminating new terms in level $r$ will not lead to recreating the already eliminated terms in previous levels. This idea of using the kernel of the maps is one of the most common approach in the unique normal form theory. However, it is evident that not only kernel terms are useful but also some terms under which the degenerate spaces are kept invariant can be used. This, of course, may not lead to further simplification of the system (since it, instead, makes a compromise on the degenerate spaces) but it leaves us with more freedom in the choice of parametric normal form which can be very important in applications. An alternative parametric normal form ({\em i.e,} equation (\ref{stab}) compared) to the equation (\ref{inftyextend}) may evidently be more suitable for bifurcation and stability analysis.

\section{A distorted spectral sequence and invariant degenerate spaces}\label{alter}

In this section we modify the method of spectral sequence to a similar approach based on the notion of degenerate invariant spaces in order to obtain our desired parametric normal form. Since the main property of the spectral sequence fails in this section, the sequence of spaces $\widehat{E}^{m,n}_r$ is not called the spectral sequence; $\widehat{E}^{m,n}_r$ is not necessarily the cohomology of $\widehat{E}^{m,n}_{r-1}.$

Any degenerate condition puts some more restrictions on each level of normal forms than what complement spaces does, {\em e.g.,} the homogenous terms of $\widetilde{v}^{(r)}$ can not freely take any vector in the $r$th level complement spaces. Mathematically, this lack of freedom can be (fully or partially) translated in terms of invariant degenerate spaces $\mathfrak{D}^{(p)}$ (\ie $\mathfrak{D}^{(p)}\subseteq \mathfrak{D}^{(q)}, p\geq q\geq r,$ for some $r\in \NZ)$ when for a sufficiently large natural number $r,$ $\widetilde{v}^{(p)}=\sum^\infty_{n=0}\widetilde{v}^{(p)}_n \in \mathfrak{D}^{(p)}= {\oplus}^\infty_{n=0}\mathfrak{D}_n^{(p)}\subseteq {\oplus}^\infty_{n=0}\NC_n^{(p)}$ $\forall p\geq r.$ The degenerate conditions (imposed by the codimension of the system) are usually distinguished when the normal form of a given system is computed up to high enough levels. For instance, the first (and higher) level extended partial (parametric) normal form of a system associated with codimension $N_0$ generalized Hopf singularity (in polar coordinates) does not have amplitude term of grade {\em less} than $2N_0+1$ ({\em i.e.,} $\rho^{i}$ for $i<2N_0+1);$ this is because of its codimension, see Lemma \ref{firstlevel}. The first level complement spaces of such system, however, have amplitude terms of all odd orders. So, we assume that the first level degenerate spaces associated with grades less than $2N_0+1$ do not have amplitude terms.
Generally we may define the $r$th level degenerate space associated with a set of conditions, {\em i.e.,} degenerate conditions, (when the style, grading structures, and approach are already fixed) by the vector space span $\mathfrak{D}^{(r)}$ of all $r$th level extended partial parametric normal forms of all systems satisfying degenerate conditions. According to the grading structure, we have $\mathfrak{D}^{(r)}\subseteq {\oplus}^\infty_{n=0} \pi_{\LS_n}\mathfrak{D}^{(r)}= {\oplus}^\infty_{n=0} \mathfrak{D}_n^{(r)},$ where $\pi_{\LS_n}\mathfrak{D}^{(r)}= \mathfrak{D}_n^{(r)}.$ Thus, we further assume $\mathfrak{D}^{(r)}= {\oplus}^\infty_{n=0} \mathfrak{D}_n^{(r)}$ is satisfied and call $\mathfrak{D}_n^{(r)}$ by the $r$th level degenerate space of grade $n.$

The computation of normal forms does not necessarily require (indeed, not a feasible approach in practical computations as Murdock \cite{Mord04} stated) the evaluation of $r$th level extended partial normal forms. Indeed, the proper approach is to evaluate the spaces $E^{n, -n+1}_r$ ($\forall r, 1\leq r\leq n$) and $E^{n, -n}_r$ $(\forall r, r\leq n-1).$ Then, by Lemma \ref{decom}, it is easy to prove that there exist $Y_r$ $(\forall r, r\leq n-1)$ such that $\phi_{Y_{n-1}}\cdots \phi_{Y_1} \phi_{Y_0}$ transforms the system into a new system in which the terms of grade $n$ be an element of $\D_n^{(n)},$ while terms with grades less than $n$ remain unchanged. The later, however, is not necessarily true in this section, since the notion of invariant degenerate spaces is used. In fact, the terms of grades less than $n$ may be changed in the $n$th step, but yet they stay within the degenerate spaces and thus it does not hamper our computational process. Therefore, in the modified approach below we just need to satisfy some conditions and then, evaluate ${\rm Total}^{1} (\widehat{E}_{\infty}^{*,*})$ which determines the parametric normal forms. We denote the new spaces with $\widehat{E}_r^{*,*},$ $\widehat{E}^{n,-n+1}_r= \frac{\NC_n^{(r)}+ \ta_n^{(r)}}{\ta_n^{(r)}}= \NC_n^{(r)}$ and $\D_n^{(r)}\subseteq \NC_n^{(r)}.$

Let us denote $\R^{(0)}_n$ for a time space satisfying $E_0^{n, -n}= \frac{(\R^{(0)}_n+\F^{n+1}\R)\times \F^n\PM\times \F^n\LS}{\F^{n+1} A^{0,*}}$ and consider its decomposition as $\R^{(0)}_n=\widetilde{\R}_n^{(0)}\oplus \widehat{\R}_n^{(0)}$ such that $\widetilde{\R}_n^{(0)}\D_0^{(1)} \subseteq \D_n^{(n)}$ and $\widehat{\R}_n^{(0)}$ is the unique complement for $\widetilde{\R}_n^{(0)},$ where the formal basis costyle is used. Our idea is to preserve the time subspace $\widetilde{\R}_n^{(0)}$ for later use and only apply the space $$\widehat{E}_0^{n, -n}= \widehat{\R}_n^{(0)}\times \PM_n\times \LS_n + \F^{n+1} A^{0, *}/\F^{n+1} A^{0,*}$$ for the first level spaces. This provides us enough flexibility and freedom to obtain our desired normal form. Therefore, $$\widehat{E}_0^{n, -n+1}= E_0^{n, -n+1}\hbox{ and }\widehat{E}^{n,-n+1}_1= \frac{\F^{n}\LS}{\F^{n+1} \LS+ d_0(\widehat{\R}_n^{(0)}\times \PM_n\times \LS_n)}.$$
Now we intend to use the time space $\widetilde{\R}_n^{(0)}$ for higher level spaces, thus we add this space to our conormal form spaces (transformation spaces). Then, we assume
\begin{equation*}
(\widetilde{\R}_n^{(0)}, \0, \mathbf{0})+ \F^{n} A^{0, *}\cap d_0^{-1}\F^{n+1}\LS= \R_n^{(1)}\! \times \!{\PM_n}^{(1)}\!\times \!\LS_n^{(1)}\,(\hbox{mode } \F^{n+1} A^{0, *}\cap d_0^{-1}\F^{n+1}\LS).
\end{equation*} Thus, $\R_n^{(0)}\subseteq\R_n^{(1)}$ denotes the time space available to be used for the second level spaces. However, we may again wish to preserve a subspace $\widetilde{\R}_n^{(1)}\subseteq \R_n^{(1)}$ for our later use provided that $\widetilde{\R}_n^{(1)}\D_i^{(2)} \subseteq \D_{n+i}^{(n+i)} \hbox{ for }i=0, 1.$ Then, consider the unique time space decomposition $$\R_n^{(1)}= \widetilde{\R}_n^{(1)}\oplus \widehat{\R}_n^{(1)}$$ obtained via formal basis costyle. Thereby, we let
\begin{equation*}
\widehat{E}^{n,-n}_1= \frac{\widehat{\R}_{n}^{(1)}\times {\PM_n}^{(1)}\times \LS_n^{(1)}+ \F^{n+1} A^{0, *}\cap d^{-1}\F^{n+1}\LS}{\F^{n+1} A^{0, *}\cap d^{-1}_1\F^{n+1}\LS}
\end{equation*}
and $$\widehat{E}^{n,-n+1}_2= \frac{\F^{n}\LS}{\F^{n+1} \LS+\pi_{\F^n \LS}d_1(\widehat{\R}_{n-1}^{(1)}\times {\PM_{n-1}}^{(1)}\times \LS_{n-1}^{(1)})+ d_0(\widehat{\R}_n^{(0)}\times \PM_n\times \LS_n)}.$$ Inductively, for any natural number $r$ we add the time space $\widetilde{\R}_n^{(r-1)}$ to the transformation spaces and denote
\begin{equation*}
(\widetilde{\R}_n^{(r-1)}, \mathbf{0}, \mathbf{0})+ \F^{n} A^{0, *}\cap d_{r-1}^{-1}\F^{n+r}\LS= \R_n^{(r)}\times {\PM_n}^{(r)}\times \LS_n^{(r)}
\end{equation*} $(\hbox{mode }\F^{n+1} A^{0, *}\cap d_{r-1}^{-1}\F^{n+r}\LS).$ Then, we choose the unique formal basis costyle time space decomposition
$\R_n^{(r)}= \widetilde{\R}_n^{(r)}\oplus \widehat{\R}_n^{(r)}$ such that
\begin{equation}\label{modcondition}
\widetilde{\R}_n^{(r)} \D_k^{(k)}\subseteq \D_{n+k}^{(n+k)}\quad \forall k, k<r,
\end{equation} with the intention of preserving the space $\widetilde{\R}_n^{(r)}$ for computing higher level spaces. Thus,
$$\widehat{E}^{n,-n}_r= \frac{\widehat{\R}_n^{(r)}\times {\PM_n}^{(r)}\times \LS_n^{(r)}+ \F^{n+1} A^{0, *}\cap d_{r-1}^{-1}\F^{n+r}\LS}{\F^{n+1} A^{0, *}\cap d_{r-1}^{-1}\F^{n+r}\LS}$$ and
\begin{equation*}
\widehat{E}^{n,-n+1}_{r+1}= \frac{\F^{n}\LS}{\pi_{\LS_n}\sum^r_{k=0}d_{k}(\widehat{\R}_{n-k}^{(k)}\times {\PM_{n-k}}^{(k)}\times \LS_{n-k}^{(k)})+ \F^{n+1} \LS}.
\end{equation*}

We apply the above approach to parametric generic $N_0$-codimension generalized Hopf singularity to obtain an alternative parametric normal form. The new normal form is more suitable for its applications in perturbation and bifurcation analysis. Define
\begin{equation}\label{sevom}
\D_{N}^{(N)}= \left\{
\begin{array}{ll}
\Span\{Y_{10}\}& \hbox{when } \quad N=0,\\[1.0ex]
{\rm span}_{\mathbb{F}}\{X_{(N_0+1)N_0}, Y_{(N_0+1)N_0}\} & \hbox{when  }\quad N=2N_0, \\[1.0ex]
{\rm span}_{\mathbb{F}}\{Y_{(N_0+1)N_0}\mu^{\mathbf{n}}\} & \hbox{when }\quad N= 2N_0+ 2rN_0+r,|\m|=r\neq 0,\\[1.0ex]
{\rm span}_{\mathbb{F}}\{X_{i(i-1)} \mu_{k}\} & \hbox{when  }\quad N= 2i+ 2N_0-1, i\leq N_0, k\leq m, \\[1.0ex]\{\mathbf{0}\}& \hbox{otherwise,}
\end{array}\right.
\end{equation} and
\begin{equation}\label{Rdecom}
\widetilde{\R}_{n}^{(i)}= \left\{
\begin{array}{ll}
\Span \{Z_{N_0}\mu^{\mathbf{n}}|n= N_0+ r\alpha, \mathbf{n}\in \NZ^m\} & \hbox{when  }\quad i<2N_0, \\[1.0ex]
\{0\}& \hbox{otherwise.}
\end{array}\right.
\end{equation} It is easy to check that $\widetilde{\R}_{n}^{(i)}$ and $\D_{N}^{(N)}$ satisfy the condition given in Equation (\ref{modcondition}). In the following the above method is implemented to compute ${\rm Total}^{1} (\widehat{E}_{\infty}^{*,*})= {\oplus}^\infty_{n=0} \NC_n^{(n)}$ $(\D_n^{(n)}\subseteq \NC_n^{(n)}).$ Clearly, Equation (\ref{Rdecom}) results in $$\R_n=\widehat{\R}_n^{(0)}, \widetilde{\R}_{n}^{(0)}=\{0\} \quad(\forall n< 2N_0), $$ and for any $i<2N_0$ we have $\widehat{\R}_{n}^{(i)}= \{0\}$ and $$\widetilde{\R}_{n}^{(i)}= \Span \{Z_{N_0}\mu^{\mathbf{n}}|n= 2N_0+ r\alpha, \mathbf{n}\in \NZ^m\}.$$ Furthermore,  $\widehat{\R}_{n}^{(2N_0)}= \Span \{Z_{N_0}\mu^{\mathbf{n}}|n= 2N_0+ r\alpha, \mathbf{n}\in \NZ^m\}.$

\begin{lem}\label{firstmodif}
Let $v,$ given by equation (\ref{originalvect}), be a parametric generic Hopf singularity system of parametric dimension $N_0,$ $\alpha= 2N_0+1,$ and $\widetilde{\R}_{n}^{(i)}$ satisfy the equation (\ref{Rdecom}). Then, we have $\widehat{E}^{n,-n+1}_1= \frac{\NC_n^{(1)}+ \ta_n^{(1)}}{\ta_n^{(1)}}= \NC_n^{(1)},$ where
\begin{eqnarray*}
\NC_n^{(1)}&=& \Span \{X_{(N_0+1)N_0}\mu^{\mathbf{n}}, Y_{(N_0+1)N_0}\mu^{\mathbf{n}}| n= 2N_0+ r\alpha, \mr\in \NZ^m\}\\
&&+\, \Span\{X_{(k+1)k}\mu^{\mathbf{n}}| n= 2k+ r\alpha, k\neq N_0, \mr\in \NZ^m\}.
\end{eqnarray*}
Furthermore, $\widehat{E}^{n,-n+1}_1= \widehat{E}^{n,-n+1}_r$ $\forall r<2N_0.$
\end{lem}

\begin{proof}
Since $v$ is a $N_0$-parametric dimension generalized Hopf singularity, $\D_n^{(1)}=\{\0\}$ $\forall n<2N_0,$ therefore, $\widehat{E}^{n,-n+1}_1= \widehat{E}^{n,-n+1}_r$ $\forall r<2N_0.$ The rest of the proof is straightforward.
\end{proof}

The following theorem presents the infinite level normal form space of an $N_0$-parametric dimension of parametric generic generalized Hopf singularity.
\begin{thm}\label{specmain}
Assume the hypothesis of Lemma \ref{firstmodif} holds. Then,
\begin{equation}\label{totalN0modif}
{\rm Total}^{1} (\widehat{E}_{2N_0}^{*,*})= {\oplus}^\infty_{n=0} \widehat{E}^{n,-n+1}_{2N_0}\cong {\oplus}^\infty_{n=0} \NC_n^{(2N_0)},
\end{equation} where
$\NC_{n}^{(2N_0)} = \{X_{(k+1)k}|n= 2k\}$ $(1\leq n< 2N_0),$ for any $n\geq 2N_0,$
\begin{eqnarray} \nonumber
\NC_{n}^{(2N_0)}&=&\Span\{\delta_{2N_0, n}X_{(N_0+1)N_0}, Y_{(N_0+1)N_0}\mu^{\mathbf{n}}| n= 2N_0+ r\alpha, \mathbf{n}\in \NZ^m\}\\
&&+\,
\Span \{X_{(k+1)k}\mu^{\mathbf{n}}| 0\leq k< N_0, n= 2k+r\alpha,\mathbf{n}\in \NZ^m\},\label{55}
\end{eqnarray} and the degenerate spaces $\D_{n}^{(2N_0)}= \D_{n}^{(n)}\subseteq \NC_{n}^{(2N_0)}$ for any $n\leq 2N_0.$ In particular $\D_{n}^{(2N_0)}=\{\0\}$ for any $0\neq n<2N_0$ and $\D_{2N_0}^{(2N_0)}=\NC_{2N_0}^{(2N_0)}.$ Furthermore, assume $v$ is parametric generic and $m=N_0$. Then, (by also using an invertible linear reparametrization on $v$) there exists a $\sigma \in S_{N_0}$ such that
\begin{equation}\label{totalinfty1}
{\rm Total}^{1} (\widehat{E}_{\infty}^{*,*})= {\rm Total}^{1} (\widehat{E}_{4N_0}^{*,*})\cong {\oplus}^\infty_{k=0} \NC_k^{(4N_0)},
\end{equation} where $\forall n\geq 2N_0,$
\begin{eqnarray*}
\NC_{n}^{(\infty)}&=&\Span\{\delta_{2N_0, n}X_{(N_0+1)N_0}, \delta_{2N_0+ r\alpha, n} Y_{(N_0+1)N_0}\mu^{\mathbf{n}}| r=|\mathbf{n}|, \mathbf{n}\in \NZ^{N_0}\}
\\
&&+\,
\Span \{X_{k(k-1)}\mu_{\sigma(i)}| 1\leq i\leq N_0, 0\leq k< N_0, n= 2k-2+\alpha\},
\end{eqnarray*} and $\D_{n}^{(\infty)}=\NC_{n}^{(\infty)}$ $\forall n\geq 2N_0.$
\end{thm}

The next theorem follows Theorem \ref{specmain}.

\begin{thm}\label{prelema1} Let $v^{(0)}\in \LS,$ given by equation (\ref{originalvect}) and $m=N_0$, be a parametric generic system associated with generalized Hopf singularity of parametric dimension $N_0.$ Then, there exist a $\sigma \in S_{N_0},$ a parametric state solution $Y^P,$ a parametric time solution $Y^T$ and a parametric state solution $Y^S$ such that their associated near-identity maps (with an invertible linear reparametrization) transform
$v^{(0)}$ to an infinite level parametric normal form:
\begin{eqnarray*}\label{normalform}
v^{(\infty)}&=&Y_{10}+ \sum^{N_0-1}_{i=0}X_{(i+1)i}\mu_{\sigma(i+1)}+a^{(2N_0)}_{(N_0+1)N_0}X_{(N_0+1)N_0}\\\nonumber
&&+ \sum^\infty_{r=|\mathbf{n}|=0}\sum_{\mathbf{n}\in \mathbb{N}^{N_0}_0}b^{(2N_0+2rN_0+r)}_{(N_0+1)N_0\mathbf{n}}Y_{(N_0+1)N_0}\mu^{\mathbf{n}},
\end{eqnarray*}
where the coefficients are uniquely expressed in terms of the coefficients of $v^{(0)}.$
\end{thm}

The following corollary provides a different parametric normal form for generalized Hopf singularity from the ones in \cite[Theorem 3]{PYuChen} and \cite{GazorYu,yl} but identical with the parametric normal form presented in \cite{Gazor,GazorYuGen}. This, of course, is a consistent alternative form with those of \cite{GazorYu,PYuChen,yl}.
\begin{cor}\label{polarhopf}
Let the parametric generic Hopf singularity system
\begin{equation}\label{hopf1}
\left(\begin{array}{c}
\frac{d{x}}{dt}\\
\frac{d{y}}{dt}\\
\end{array}\right) = \left(
\begin{array}{c}y\\-x\\\end{array}\right)
+\sum_{i+j+|\m|=2, i+j\geq 1}^\infty\left(\begin{array}{c}
\alpha_{ij\mathbf{n}}\\
\beta_{ij\mathbf{n}}\\
\end{array}
\right)x^iy^j\mu^{\mathbf{n}},
\end{equation} where $\mu\in \mathbb{R}^{N_0},$ have a parametric dimension of $N_0.$ Then, there exist a $\sigma \in S_{N_0},$ a sequence of near-identity change of state variables, time rescaling and reparametrization maps (near identity as well as an invertible linear reparametrization) such that system (\ref{hopf1}) can be transformed to an infinite level parametric normal form (given in polar coordinates):
\begin{eqnarray}\label{stab}
\frac{d{\rho}}{dt}
&=& \rho \left[ A
\rho^{2N_0}+ \sum^{N_0}_{i=1}\rho^{2i-2}
\mu_{\sigma(i+1)} \right],\\\nonumber
\frac{d{\theta}}{dt}&=& 1+\rho^{2N_0}\sum^\infty_{|\m|=0}\sum_{\mathbf{n}\in \mathbb{N}^{N_0}_0}B_{
\mathbf{n}}
\mu^{\mathbf{n}},
\end{eqnarray}
where the coefficients $A
$ and $B_{
\mathbf{n}}$ are uniquely determined in terms of $\alpha_{ij\mathbf{n}}$ and $\beta_{ij\mathbf{n}}.$
\end{cor}
\begin{proof} The corollary readily follows the proof of \cite[Corollary 4.3]{GazorYu} in the light of Theorem \ref{prelema1}.
\end{proof}

\section{Different approaches on normal forms of Hopf singularity}\label{compare}

In this section, we briefly present different simplest normal forms (parametric and also nonparametric) in polar coordinates that can be obtained by using or not using time rescaling and reparametrization.

We consider the parametric system given by the Equation (\ref{hopf1}) whose parametric dimension is $N_0.$ With only (near identity change of parametric) state maps without time rescaling and reparametrization, the following parametric normal form can be obtained:
\begin{eqnarray}\label{onlystate}
v^{(\infty)}&=&(1+\sum^\infty_{|\m|=0}\sum_{\mathbf{n}\in \mathbb{N}^m_0}\sum^{N_0}_{l=1}b_{(l+1)l\mathbf{n}_1}\rho^{2l}\mu^{\mathbf{n}})\partial_\theta+ (\sum^\infty_{|\m|=1}\sum_{\mathbf{n}\in \mathbb{N}^m_0} a_{(2N_0+1)2N_0\m}\rho^{4N_0+1}\mu^\m\\\nonumber&&+a_{(N_0+1)N_0}\rho^{2N_0+1}+ \sum^\infty_{|\m|=1}\sum_{\mathbf{n}\in \mathbb{N}^m_0}\sum^{N_0-1}_{i=0}a_{(i+1)i\mathbf{n}_{r}}
\rho^{2i+1}\mu^{\mathbf{n}})\partial_\rho.
\end{eqnarray} This is while the state maps with reparametrization (including an invertible linear reparametrization) can help simplifying Equation (\ref{onlystate}) further to
\begin{eqnarray}\label{onlystatepar}
v^{(\infty)}&=&(1+\sum^\infty_{|\m|=0}\sum_{\mathbf{n}\in \mathbb{N}^m_0}\sum^{N_0}_{l=1}b_{(l+1)l\mathbf{n}}\rho^{2l}\mu^{\mathbf{n}})\partial_\theta+ (\sum^\infty_{r=1}\sum_{\mathbf{n}\in \mathbb{N}^m_0} a_{(2N_0+1)2N_0\m}\rho^{4N_0+1}\mu^\m\\\nonumber&&+a_{(N_0+1)N_0}\rho^{2N_0+1}+ \sum^{N_0-1}_{i=0}
\rho^{2i+1}\mu_i)\partial_\rho.
\end{eqnarray}

The orbital equivalence for parametric normal forms (i.e., the state maps and time rescaling are only used without any usage of reparametrization) follows the equation
\begin{eqnarray}\label{timestate}
v^{(\infty)}\!&\!=\!&\!(a_{(N_0+1)N_0}\rho^{2N_0+1}\!+ \!\!\sum^\infty_{|\m|=1, \mathbf{n}}\!\!\sum^{N_0-1}_{i=0}a_{(i+1)i\mathbf{n}}
\rho^{2i+1}\mu^{\mathbf{n}})\partial_\rho\!+(1+\!\!\sum^\infty_{|\m|=0,\mathbf{n}}\!\!b_{\mathbf{n}}\rho^{2N_0}\mu^{\mathbf{n}})\partial_\theta.
\end{eqnarray} One may compare the equations (\ref{onlystate})-(\ref{timestate}) with the equation (\ref{stab}) to get an idea on how time rescaling and reparametrization may contribute in our parametric normal forms.

Finally, the following corollary presents the nonparametric normal forms for generalized Hopf singularity as well as its two alternative orbital equivalence, see also \cite{AlgabaSur,baiderchurch,pwang}.
\begin{cor}[Nonparametric systems]\label{nopar} Let $v= v^{(0)}$ represent a system of generalized Hopf singularity with no parameter. Then, there exists a natural number $N_0$ such that the system can be
transformed to
\begin{enumerate}
\item \label{nopar1} the simplest normal form in polar coordinates:
\begin{equation}\label{ronoprameter1}
\frac{d{\rho}}{dt}\partial_\rho+\frac{d{\theta}}{dt}\partial_\theta
=(A_{N_0}\rho^{2N_0+1}+A_{2N_0}\rho^{4N_0+1})\partial_\rho+ (1+\sum^{N_0}_{i=n} B_{i}\rho^{2i})\partial_\theta,
\,\qquad A_{N_0}\neq 0,
\end{equation}
where state maps are only used, or to
\item \label{nopar2} (when time rescaling and state maps are both used) the simplest orbital equivalence
\begin{equation}\label{ronoprameter2}
\frac{d{\rho}}{dt}\partial_\rho+\frac{d{\theta}}{dt}\partial_\theta
=(A_{N_0}\rho^{2N_0+1}+A_{2N_0}\rho^{4N_0+1})\partial_\rho+ \partial_\theta, A_{N_0}\neq 0,
\end{equation} or alternatively to
\begin{equation}\label{noprameter2}
\frac{d{\rho}}{dt}\partial_\rho+\frac{d{\theta}}{dt}\partial_\theta
=A_{N_0}\rho^{2N_0+1}\partial_\rho+ (1+ B_{N_0}\rho^{2N_0+1})\partial_\theta, A_{N_0}\neq 0.
\end{equation}
\end{enumerate}
\end{cor}

\section{Conclusions}\label{sec6}
The method of spectral sequences has been suitably generalized to consider
the simplest parametric normal forms of parametric vector fields. Our results can also be considered as a generalization for the spectral sequences of orbital equivalence of non-parametric systems. We also introduce a new style (and costyle) for obtaining unique normal forms. The method is applied to obtain two different parametric normal forms associated with generalized Hopf singularity.

\vspace{0.10 in}

\noindent\textbf{Acknowledgments.}
The first author would like to acknowledge Prof. J. Murdock's generous help and contributions via numerous long e-mail discussions. 



\begin{thebibliography}{00}

\bibitem{AlgabaSur} A. Algaba, E. Freire, E. Gamero,
``Characterizing and computing normal forms using Lie transforms: a survey,''
{\rm Dyn. Continuous, Discrete Impulsive Systems} 8 (2001) 449–-475.

\bibitem{Algaba} A. Algaba, E. Freire, E. Gamero, C. Garcia,
``Quasi-homogeneous normal forms,''
{\rm J. Comp. and Appl. Math.} {\bf 150} (2003) 193--216.



\bibitem{Arnold75} V.I. Arnol'd,
``A spectral sequence for the reduction of functions to normal form,''
{\rm Funkcional. Anal. i Prilo$\check{\hbox{z}}$en.} {\bf 9} (3) (1975) 81-–82.

\bibitem{Arnold76} V.I. Arnol'd,
Spectral sequences for the reduction of functions to normal forms,
{\rm S.L. Dobolev, A.I. Suslov (Eds.), Problems in Mechanics and Mathematical Physics, Izdat.
''Nauka''}, Moscow, 1976, pp. 7–-20, 297 (Russian).

\bibitem{baider} A. Baider,
``Unique normal forms for vector fields and hamiltonians,''
{\rm J. Diff. Eqs.} {\bf 78} (1989) 33--52.


\bibitem{baiderchurch} A. Baider, R.C. Churchill,
``Unique normal forms for palnnar vector fields,''
{\it Math. Z.} {\bf 199} (1988) 303--310.



\bibitem{baidersanders} A. Baider, J.A. Sanders,
``Further reductions of the Takens-Bogdanov normal form,''
{\rm J. Diff. Eqs.} {\bf 99} (1992) 205--244.

\bibitem{belitskii} G. Belitskii,
``Smooth equivalence of germs of $C^\infty$ of vector fields with zero or a pair of pure imaginary eigenvalues,''
{\rm Funct. Anal. Appl.} {\bf 20} (1986) 253--259.

\bibitem{benderchur} M. Benderesky, R. Churchill,
``A spectral sequence approach to normal forms,''
{\rm Recent developements in algebraic topology,
Contemp. Math.} {\bf 407} Amer. Math. Soc., Providence, RI. (2006) 27--81.

\bibitem{chendora} G. Chen, J. Della Dora,
``Further Reductions of Normal Forms for Dynamical Systems,''
{\rm J. Diff. Eqs.} {\bf 166} (2000) 79--106.



\bibitem{ChowLiWang} S.N. Chow, C. Li, D. Wang,
``Normal Forms and Bifurcation of Planar Vector Fields,''
Cambridge Univ. Press, Cambridge, UK, 1994.



\bibitem{Gazor} M. Gazor,
``Spectral sequences method and computation of parametric normal forms of differential equations,''
{\rm PhD Thesis, Department of Applied Mathematics, University of Western Ontario,} 2008.

\bibitem{GazorYuGen} M. Gazor, P. Yu,
``Formal decomposition method and parametric normal form of generalized Hopf singularity,'' preprint.


\bibitem{GazorYu} M. Gazor, P. Yu,
``Infinite order parametric normal form of Hopf singularity,''
{\rm Int. J. Bifurcation Chaos} {\bf 11} (18) (2008) 3393--3408.

\bibitem{Gaeta1999} G. Gaeta,
``Poincare renormalized forms,''
{\rm Ann. Inst. H. Poincare (Phys. Theo.)} {\bf 70} (1999) 461--514.


\bibitem{kw}
H. Kokubu, H. Oka, D. Wang,
``Linear grading function and further reduction of normal forms,''
{\rm J. Diff. Eqs.} {\bf 132} (1996) 293--318.

\bibitem{Kuz} Y.A. Kuznetsov,
Elements of Applied Bifurcation Theory, 3rd edition, Springer-Verlag,
New York, 2004.


\bibitem{LiaYu} X. Liao, L. Wang, P. Yu,
Stability of Dynamical Systems,
Elsvier, 2007.

\bibitem{McCleary} J. McCleary,
A user's guide to spectral sequences,
Second edition, Cambridge; New York; Cambridge University Press, 2001.

 	
\bibitem{MurdBook} J. Murdock,
Normal Forms and Unfoldings for Local Dynamical Systems.
Springer-Verlag, New York, 2003.

\bibitem{Mord04} J. Murdock,
``Hypernormal form theory: foundations and algorithms,''
{\it J. Diff. Eqs.}  {\bf 205} (2) (2004) 424--465.


\bibitem{Mord08} J. Murdock, D. Malonza, ``An improved theory of asymtotic unfoldings,'' in press, {\it J. Diff. Eqs.}
(2009) doi:10.1016/j.jde.2009.04.014.

\bibitem{MurdSandersbox}
J. Murdock, J.A. Sanders,
``A new transvectant algorithm for nilpotent normal forms,''
{\it J. Diff. Eqs.} {\bf 238} (1) (2007) 234--256.

\bibitem{pwang} J. Peng, D. Wang,
``A suffiecient condition for the uniqueness of normal forms and unique normal forms
of generalized Hopf singularities,''
{\rm Int. J. Bifurcation Chaos} {\bf 14} (9) (2004) 3337--3345.


\bibitem{Sanders03}
J.A. Sanders,
``Normal form theory and spectral sequences,''
{\it J. Diff. Eqs.} {\bf 192} (2) (2003) 536--552.

\bibitem{Sanders05}
J.A. Sanders,
``Normal form in filtered Lie algebra representations,''
{\it Acta Appl. Math.}
{\bf 87}, (2005) 165--189.

\bibitem{SandersMurd}
J.A. Sanders, F. Verhulst, J. Murdock,
Averaging methods in nonlinear dynamical systems,
Second edition. Applied Mathematical
Sciences, 59. Springer-Verlag, New York, 2007.

\bibitem{Zoladek02} E. Str\'{o}\.{z}yna, and H. Zoladek,
``The analytic and formal normal form for the nilpotent
singularity,''
{\it J. Diff. Eqs.} {\bf 179} (2002) 479--537.

\bibitem{Zoladek03} E. Str\'{o}\.{z}yna, and H. Zoladek,
``Orbital formal normal forms for general Bogdanov–Takens
singularity,''
{\it J. Diff. Eqs.} {\bf 193} (2003) 239--259.





\bibitem{Weibel} C.A. Weibel,
An introduction to homological algebra,
Cambridge-New York, Cambridge University Press, 1994.

\bibitem{yy2002} P. Yu,
``Simplest normal forms of Hopf and generalized Hopf bifurcations,''
{\rm Int. J. Bifurcation Chaos} {\bf 9} (10) (1999) 1917--1939.


\bibitem{PYuChen} P. Yu, G. Chen
``The simplest parametrized normal forms of Hopf and generalized Hopf
bifurcations,''
{\rm Nonlinear Dynamics} {\rm 14} (50) (2007) 297--313.

\bibitem{yl} P. Yu, A.Y.T. Leung,
``The simplest normal form of Hopf bifurcation,''
{\rm Nonlinearity} {\bf 16} (1) (2003) 277--300.





\end{thebibliography}
\end{document}